\documentclass[10pt]{article}

\usepackage{amsmath}
\usepackage{url}
\usepackage{amsthm}
\usepackage{authblk}
\usepackage{amssymb}
\usepackage{amsfonts}
\usepackage{graphicx}
\usepackage{verbatim}
\usepackage{mathrsfs}
\usepackage{subfigure}
\usepackage{fancyhdr}
\usepackage{latexsym}
\usepackage{graphicx}
\usepackage{dsfont}
\usepackage{color}
\usepackage{fancybox}
\usepackage{nicefrac}
\usepackage{framed, color}
\usepackage{floatrow}
\definecolor{shadecolor}{rgb}{1, 0.8, 0.3}

\usepackage {blindtext}
\usepackage{geometry}
\usepackage{fancyhdr}
\usepackage{amsmath , amsthm , amssymb}
\usepackage{amsfonts}
\usepackage{graphicx}
\usepackage{hyperref}
\usepackage{enumitem}

\normalfont
\usepackage[T1]{fontenc}
\usepackage{relsize}
\usepackage[utf8]{inputenc}
\usepackage[english]{babel}
\usepackage{fancyhdr}
\usepackage{enumitem,graphicx,appendix}
\usepackage{amsmath,amssymb,amsthm}
\usepackage{xargs}
\usepackage{color}

\SetLabelAlign{Center}{\hfil \mbox({#1}) \hfil}
\newfloatcommand{capbtabbox}{table}[][\FBwidth]

\title{
Analytic continuation of local (un)stable manifolds with 
rigorous computer assisted error bounds}

\author[1]{William D. Kalies  \thanks{
Email: {\tt wkalies@fau.edu}}}

\author[2]{Shane Kepley \thanks{S.K. partially supported by NSF 
grants DMS-1700154 and DMS - 1318172, and by  the Alfred P. Sloan Foundation grant G-2016-7320
Email: {\tt skepley@my.fau.edu}} }

\author[3]{J.D. Mireles James \thanks{J.M.J partially supported by NSF 
grants DMS-1700154 and DMS - 1318172, 
and by  the Alfred P. Sloan Foundation grant G-2016-7320
Email: {\tt jmirelesjames@fau.edu}}}

\affil[1,2,3]{Florida Atlantic University, Department of Mathematical Sciences}
\date{\today}

\newcommand{\rr}{\mathbb{R}}

\newcommand{\cc}{\mathbb{C}}
\newcommand{\nn}{\mathbb{N}}
\newcommand{\TT}[1]{\mathcal{T} \left( #1 \right)}
\newcommand{\itt}[1]{\mathcal{T}^{-1} \left( #1 \right)}
\newcommand{\overbar}[1]{\mkern 1.5mu\overline{\mkern-1.5mu#1\mkern-1.5mu}\mkern 1.5mu}
\newcommand{\norm}[1]{\left|\left| #1 \right|\right|}
\newcommand{\normX}[1]{||#1||_{\mathcal{X}}}

\theoremstyle{remark}
\newtheorem{remark}{Remark}

\newtheorem{theorem}{Theorem}[section]
\newtheorem{lemma}[theorem]{Lemma}
\newtheorem{prop}[theorem]{Proposition}

\begin{document}

\maketitle
\begin{abstract}
We develop a validated numerical procedure
for continuation of local stable/unstable manifold patches 
attached to  equilibrium solutions 
of ordinary differential equations.
The procedure has two steps.
First we compute an accurate high order Taylor
expansion of the local invariant manifold.  This 
expansion is valid in some neighborhood of the 
equilibrium.  An important component of our method is 
that we obtain mathematically rigorous lower 
bounds on the size of this neighborhood, as well as 
validated a-posteriori error bounds for the polynomial approximation. In the second step we use a rigorous numerical integrating scheme to 
propagate the boundary of the 
local stable/unstable manifold as long as possible, i.e.\ as 
long as the integrator yields validated error bounds  
below some desired tolerance.  The procedure exploits adaptive remeshing 
strategies which track the growth/decay of the Taylor coefficients of the 
advected curve. In order to highlight the utility of the procedure we study the embedding of 
some two dimensional manifolds in the Lorenz system.  
\end{abstract}

\section{Introduction}\label{sec:intro}
This paper describes a validated numerical method for computing 
accurate, high order approximations of stable/unstable manifolds of 
analytic vector fields. Our method generates a system of polynomial 
maps describing the manifold away from the 
equilibrium. The polynomials approximate charts for the manifold, and each comes equipped with mathematically 
rigorous bounds on all truncation and discretization errors.
A base step computes a parameterized
local stable/unstable manifold valid in a neighborhood of the equilibrium point.
This analysis exploits the \textit{parameterization method}
\cite{MR1976079, MR1976080, MR2177465, MR2821596, JDMJ01, MR3207723}.
The iterative phase of the computation begins by meshing 
the boundary of the initial chart
into a collection of submanifolds.  The submanifolds
are advected using a Taylor integration scheme, again equipped 
with mathematically rigorous validated error bounds.  

Our integration scheme provides a Taylor expansion in 
both the time and space variables, but uses only the spatial variables
in the invariant manifold. This work builds on the substantial 
existing literature on validated numerics for 
initial value problems, or \textit{rigorous integrators}, see for example 
\cite{MR1652147, MR2644324, MR1930946, MR3281845},
and exploits optimizations developed in 
\cite{HLM, BDLM, fourierAutomaticDiff}.

After one step of integration
we obtain a new system of charts which describe the advected 
boundary of the local stable/unstable manifold.  
The new boundary is adaptively 
remeshed to minimize integration errors in the next step.  
The development of a mathematically rigorous
remeshing scheme to produce the new system of boundary arcs
is one of the main technical achievements of the present work,
amounting to a validated numerical verification procedure 
for analytic continuation problems in several complex variables.  
Our algorithm exploits the fact that the operation of
recentering a Taylor series can be thought of as 
a bounded linear operator on a certain 
Banach space of infinite sequences (i.e.\ the Taylor coefficients),
and this bounded linear operator can be studied by adapting 
existing validated numerical methods. 
The process of remeshing is iterated as long
as the validated error bounds are held below some user specified
tolerance, or a specified number of time units.

To formalize the discussion we introduce notation. We restrict the discussion to unstable manifolds and note that our procedure applies to stable manifolds equally well by reversing the direction of time. Suppose that $f \colon \mathbb{R}^n \to \mathbb{R}^n$ 
is a real analytic vector field, and assume that $f$ generates
a flow on an open subset $U \subset \mathbb{R}^n$.  Let 
$\Phi \colon U \times \mathbb{R} \to \mathbb{R}^n$ 
denote this flow.  

Suppose that 
$p_0 \in U$ is a hyperbolic equilibrium point with $d$ unstable 
eigenvalues.  By the unstable manifold theorem
there exists an $r > 0$ so that the set
\[
W_{\mbox{\tiny loc}}^u(p_0, f, r) := \left\{x \in B_r^n(p_0) : 
\Phi(x, t) \in B_r^n(p_0) \mbox{ for all } t \leq 0
\right\},
\]  
is analytically diffeomorphic to a $d$-dimensional disk which is tangent at $p_0$
to the unstable eigenspace of the matrix $Df(p_0)$. Moreover, $\Phi(x, t) \to p_0$ as   $t \to -\infty$ for each $x \in W_{\mbox{\tiny loc}}^u(p_0, f, r)$.
Here $B_r^n(p_0)$ is the ball of radius 
$r > 0$ about $p_0$ in $\mathbb{R}^n$.  
We simply write $W_{\mbox{\tiny loc}}^u(p_0)$ when $f$ and $r$
are understood. \textit{The unstable manifold} is then defined as the collection of all points $x \in \mathbb{R}^n$ such that
$\Phi(x, t) \to p_0$ as $t \to - \infty$ which is given explicitly by
\[
W^u(p_0) = \bigcup_{0 \le t } \Phi \left( W_{\mbox{\tiny loc}}^u(p_0), t \right).
\]

The first step of our program is to compute an analytic chart map for the local manifold of the form, $P \colon \overbar{B_1^d(0)} \to \mathbb{R}^n$, such that $P(0) = p_0$, $\mbox{image}(DP(0))$ is contained the unstable eigenspace, and
\[
\mbox{image}(P) \subset W_{\mbox{\tiny loc}}^u(p_0).
\]
In Section \ref{sec:parameterization} we describe how this is done rigorously with computer assisted a-posteriori error bounds. 

Next, we note that $W^u_{\mbox{\tiny loc}}(p_0)$ is backward invariant under $\Phi$, 
and thus the unstable manifold is the forward image of the boundary of the local unstable manifold by the flow. To explain how we exploit this, suppose we have computed the chart of the local manifold described above. 
\begin{figure}[t!]
\center{
  \includegraphics[width=0.8\linewidth]{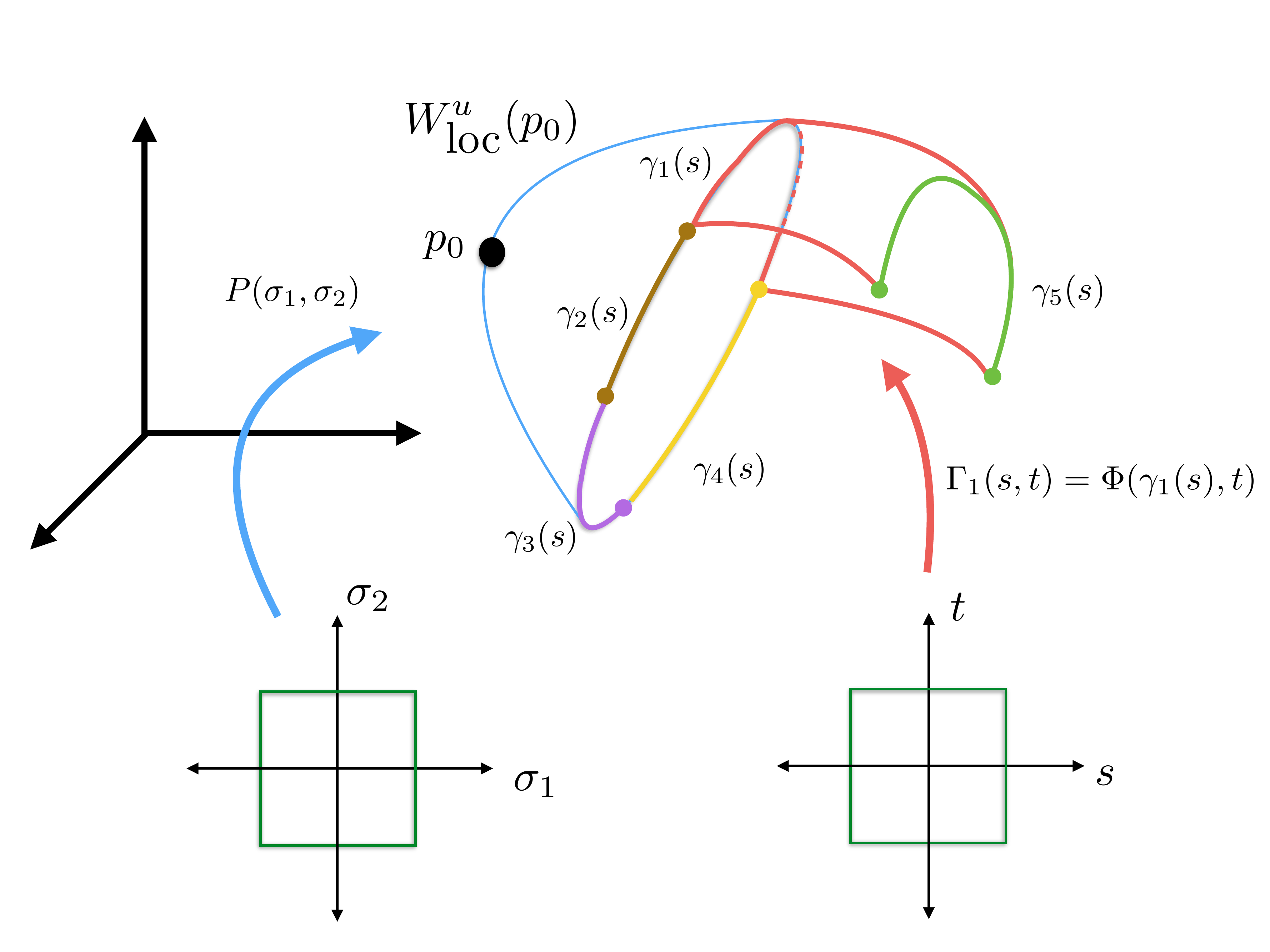}
}
\caption{The figure provides a schematic rendering of the two kinds of 
charts used on our method.  Here $P$ is the local patch containing the 
fixed point. This chart is computed and analyzed using the parameterization
method discussed in Section \ref{sec:parameterization}.  
The boundary of the image of $P$ is meshed into a number of 
lower dimensional patches $\gamma_j(s)$ and the global manifold is ``grown'' by 
advecting these patches. This results in the charts $\Gamma_j(s,t)$
describing the manifold far from the equilibrium point.} \label{fig:introSchematic}
\end{figure}
We choose a piecewise analytic system of functions
$\gamma_j \colon \overbar{B_1^{d-1}(0)} \to \mathbb{R}^n$, 
$1 \leq j \leq K_0$, such that 
\[ 
\bigcup_{1 \leq j \leq K_0} \gamma_j\left(\overbar{B_1^{d-1}(0)}\right) = \partial P(B^d_1(0)),
\]
with
\[
\mbox{image}(\gamma_i) \cap \mbox{image}(\gamma_j) 
\subset \partial\, \mbox{image}(\gamma_i) \cap \partial \,\mbox{image} (\gamma_j),
\]
i.e.\ the functions $\gamma_j(s)$, $1 \leq j \leq K_0$ parameterize the boundary of 
the local unstable manifold, and their pairwise intersections are $(d-2)$-dimensional submanifolds. Now, fix a time $T > 0$, and for each $\gamma_j(s)$, $1 \leq j \leq K_0$, define $\Gamma_j \colon \overbar{B_1^{d-1}(0)} \times [0, T] \to \mathbb{R}^n$ by 
\[
\Gamma_j(s,t) = \Phi(\gamma_j(s), t) \qquad (s,t) \in \overbar{B_1^{d-1}(0)} \times [0, T].
\]
We note that
\[
\mbox{image}(P) \cup \left( \bigcup_{1 \leq j \leq K_0} \mbox{image}(\Gamma_j) \right)
\subset W^u(p_0),
\]
or in other words, the flow applied to the boundary of the local unstable manifold yields a larger piece of the unstable manifold. Thus, the second step in our program amounts to rigorously computing the charts $\Gamma_j$ and is described in Section \ref{sec:rigorousIntegration}.  
Figure \ref{fig:introSchematic} provides a graphical illustration of the 
scheme. 
\begin{figure}[t!]
	\includegraphics[width=1.075\linewidth]{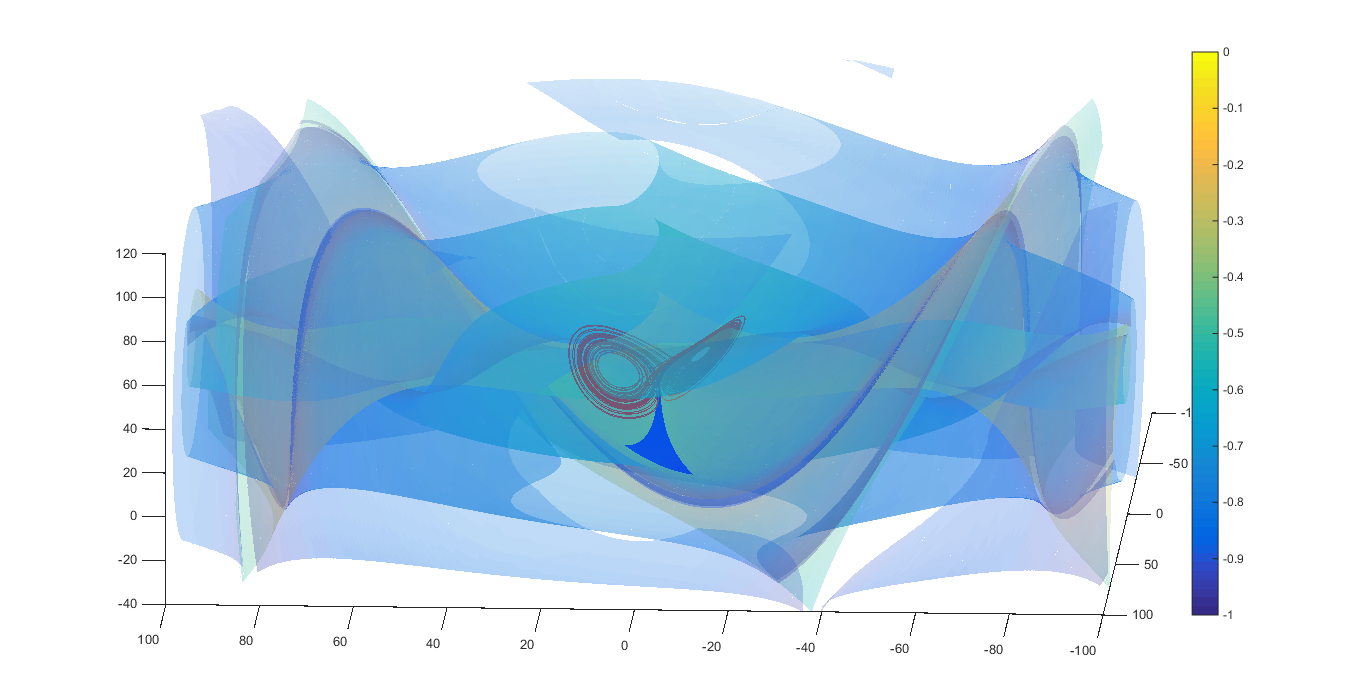}
	\caption{A validated two dimensional local stable 
		manifold of the origin in the Lorenz system at the classical parameter values:
		The initial local chart $P$ is obtained using the parameterization method, 
		as discussed in Section \ref{sec:parameterization}, and describes the 
		manifold in a neighborhood of the origin. The local stable manifold
		is the dark blue patch in the middle of the picture, below the attractor. A reference orbit near the attractor is shown in red for context. The boundary of the image of $P$ is meshed into arc segments and the global manifold is computed by advecting arcs by the flow using the rigorous integrator discussed in Section \ref{sec:rigorousIntegration}. The numerical details for this example are provided in Section \ref{sec:Lorenzresults}.} 
	\label{fig:slowstable}
\end{figure}

Figure \ref{fig:slowstable} illustrates the results of our method in a specific example.
Here we advect the boundary of a high order parameterization of the local stable  manifold at the origin of the Lorenz system at the classical parameter values, see Section~\ref{sec:parameterization}. The color of each region of the manifold describes the integration time $t \in [-1,0]$. The resulting manifold is described by an atlas consisting of 4,674 polynomial charts computed to order 24 in time and 39 in space. The adaptive remeshing described in Section \ref{sec:remesh} is performed to restrict to the manifold bounded by the rectangle $[-100,100] \times [-100,100] \times [-40,120]$. 
\begin{remark}[Parameterization of local stable/unstable manifolds]
\label{rem:introParmMethod}
Validated numerical algorithms for solving initial value problems are computationally intensive, 
and it is desirable to postpone as long as possible the moment 
when they are deployed. In the present applications we would like to 
begin with a system of boundary arcs which are initially 
as far from the equilibrium as possible, so that the efforts
of our rigorous integrator are not spent recovering the approximately 
linear dynamics on the manifold.  To this end, we employ a high order
polynomial approximation scheme based on the parameterization method
of \cite{MR1976079, MR1976080, MR2177465}.  For our purposes
it is important  to have also mathematically rigorous error 
bounds on this polynomial approximation, and here we exploit a-posteriori 
methods of computer assisted proof for the parameterization method
developed in the recent work of
\cite{MR3068557, MR2821596, parmChristian, maximeJPMe, BDLM}.
These methods yield bounds on the 
errors and on the size of the domain of analyticity, accurate to nearly machine precision,
even a substantial distance from the equilibrium.
See also the lecture notes \cite{jayAMSnotes}.
\end{remark}

\begin{remark}[Technical remarks on validated numerics for initial value problems] \label{rem:rogorousIntegrators}
A thorough review of the literature, much less any serious comparison 
of existing rigorous integrators, are tasks far beyond the scope of the present work.
We refer the interested reader to the discussion in the recent 
review of \cite{MR3504856}.  That being said, a few brief remarks on some
similarities and differences between the present and existing works are
in order.   The comments below reflect the fact that different studies 
have differing goals and require different 
tools: our remarks in no way constitute a criticism of any existing
method.  The reader should keep in mind that our goal 
 is to advect nonlinear sets of initial conditions
 which are parameterized
by analytic functions.

In one sense our validated integration scheme is closely related to 
that of \cite{MR1652147}, where rigorous  
Taylor integrators for nonlinear sets of initial conditions are developed.
A technical difference is that the a-posteriori error analysis implemented in 
\cite{MR1652147} is based on an application of the Schauder Fixed Point Theorem to a
Banach space of continuous functions. The resulting
error bounds are given in terms of continuous rather than analytic functions.

In this sense our integration scheme is also related to the work of 
\cite{MR3022075, MR3281845} on Taylor integrators 
in the analytic category.  While the integrators in the works just cited are
used to advect points or small boxes of initial conditions, the authors expand 
the flow in a parameter as well as in time, 
validating expansions of the flow in several complex variables.
A technical difference between the method employed in this work
and the work just cited is that our a-posteriori 
analysis is based on a Newton-like method, rather than the contraction mapping 
theorem.

The Newton-like analysis applies to  polynomial approximations 
which are not required to have interval coefficients.  Only the 
bound on the truncation error is given as an interval.  The truncation 
error in this case is not a tail, as the unknown analytic function may 
perturb our polynomial coefficients to all orders.  We only know that this 
error function has small norm.

This can be viewed as an analytic version of the 
``shrink wrapping'' discussed in \cite{MR2312532}. However, in our case
the argument does not lose control of bounds on derivatives. 
Cauchy bounds can be used to estimate derivatives of the truncation
error, after giving up a small portion of the 
validated domain of analyticity.  Such techniques have been used before in
the previous work of  \cite{parmChristian, maximeJPMe}. The works just cited
deal with Taylor methods for invariant manifolds rather than rigorous integrators.

Since our approach requires only floating point
rather than interval enclosures of Taylor coefficients,  
we can compute coefficients using a numerical Newton scheme
rather than solving term by term using recursion.  
Avoiding recursion can be advantageous when computing a large number of 
coefficients for a multivariable series.  The quadratic convergence
of Newton's method facilitates rapid computation to high order. 
Note also that while our method does require the inversion of a large matrix, this matrix is upper triangular,
and hence this inversion can be managed fairly efficiently.  

Any discussion of rigorous integrators must mention the work of the CAPD group.
The CAPD library is probably the most sophisticated and widely used 
software package for computer assisted proof in the dynamical systems
 community.  The interested reader will want to consult the works of 
\cite{MR1930946, cnLohner}.  The CAPD algorithms are based on the pioneering 
work of Lohner \cite{MR950221, MR904317, MR1387154}, and instead of
using fixed point arguments in function space to manage truncation errors,
develop validated numerical bounds based on the Taylor remainder theorem.
The CAPD algorithms provide results in the $C^k$ category, and are 
often used in conjunction with topological arguments in a Poincare section  
\cite{MR1626596, MR2271217, MR2012847, MR1961956, MR3032848, MR3443692}
to give computer assisted proofs in dynamical systems theory.
\end{remark}

\begin{remark}[Basis representations for analytic charts]
\label{rem:basis}
In this work we describe our method by computing charts for both the local parameterization and its advected image using Taylor series (i.e.\ analytic charts are expressed in a monomial basis). This choice allows for ease of exposition and implementation. However, the continuation method developed here works in principle 
for other choice of basis.  What is needed is a method for rigorously computing error estimates. 

Consider for example the case of an (un)stable manifold attached to a periodic orbit
of a differential equation.  In this case one could parameterize the local manifold using a 
Fourier-Taylor basis as in  basis as in 
\cite{MR2551254, MR3118249, doi:10.1137/140960207}.
Such a local manifold could then be continued using Taylor basis for the rigorous integration
as discussed in the present work. Alternatively, if one is concerned with obtaining the largest globalization of the manifold with minimal error bounds it could be appropriate to use a Chebyshev basis for the rigorous integration to reduce the required number of time steps. The point is that we are free to choose any appropriate basis for the charts in space/time provided it is amenable to rigorous validated error estimates. The reader interested in 
computer assisted proofs compatible with the presentation of the present work -- and 
using bases other than Taylor -- are referred to 
\cite{HLM,BDLM,MR2679365,dlLFGL,fourierAutomaticDiff, jpPO_PDE}
\end{remark}
\begin{remark}[Why continue the local manifold?]
As just mentioned there are already many studies in the literature which give validated 
numerical computations of local invariant manifolds, as well as computer assisted
proofs of the existence of connections between them.  Our methods provide another  
approach to the computer assisted study of connecting orbits via the ``short connection'' mechanism
developed in \cite{MR3207723}.  But if one wants to rule out other connections then it is necessary 
to continue the manifold, perhaps using the methods of the present work.
Correct count for connecting orbits is essential for example in applications concerning optimal 
transport time, or for computing boundary operators in Morse/Floer homology theory.  
\end{remark}
\begin{remark}[Choice of the example system]
The validated numerical theorems discussed in the present work are benchmarked
for the Lorenz system.  This choice has several advantages, which we explain briefly.  
First, the system is three dimensional with quadratic
nonlinearity. Three dimensions facilitates drawing of nice pictures
which provide useful insight into the utility of the method. 
The quadratic nonlinearity minimizes 
technical considerations, especially the derivation of certain analytic estimates.  
We remark however that the utility of the Taylor methods discussed here are by no means limited 
to polynomial systems.  See for example the discussion of automatic differentiation in 
\cite{mamotreto}.  We note also that many of the computer assisted proofs discussed in the preceding 
remark are for non-polynomial nonlinearities.  The second and third authors of the present work
are preparing a manuscript describing computer assisted proofs of chaotic motions 
for a circular restricted four body problem which uses the methods of the 
present work.

Another advantage of the Lorenz system is that we exploit the discussion of
rigorous numerics for stable/unstable manifolds given in the Lecture notes of 
\cite{jayAMSnotes}.  Again this helps to minimize technical complications 
and allows us to focus instead on what is new here.

Finally, the Lorenz system is an example where other authors have conducted 
some rigorous computer assisted studies growing invariant manifolds attached to 
equilibrium solutions of differential equations.  The  reader wishing to  
make some rough comparisons between existing methods 
might consult the Ph.D. thesis \cite{wittigThesis}, see especially Section $5.3.5.2$.
For example one could compare the results illustrated in Figure $5.18$ of that
Thesis with the results illustrated in Figure $2$ of the present work.  The
manifolds in these figures have comparable final validated error bounds, 
while the manifold illustrated in Figure $2$ explores a larger region of 
phase space. 

We caution the reader that such comparisons must be made only cautiously.
For example the validation methods developed in \cite{wittigThesis} are based 
on topological covering relations and cone conditions, 
which apply in a $C^2$ setting. Hence the methods of 
\cite{wittigThesis} apply in a host of situations
where the methods of the present work -- which are based on 
the theory of analytic functions of several complex variables -- breakdown.  Moreover 
the initial local patch used for the computations in \cite{wittigThesis}
is smaller than the validated local manifold developed in   
\cite{jayAMSnotes} from which we start our computations.   
\end{remark}

\bigskip

The remainder of the paper is organized as follows.
In Section \ref{sec:background} we recall some basic facts from 
the theory of analytic functions of several complex variables, 
define the Banach spaces of infinite sequences used throughout the 
paper, and state an a-posteriori theorem used in later sections.
In Section \ref{sec:parameterization} we review the parameterization 
method for stable/unstable manifolds attached to equilibrium solutions
of vector fields.  In particular we illustrate the formalism which leads to 
high order polynomial approximations of the local invariant manifolds for the 
Lorenz system, and state an a-posteriori theorem which provides the 
mathematically rigorous error bounds.  Section 
\ref{sec:rigorousIntegration} describes in detail the subdivision strategy for 
remeshing analytic submanifolds and the rigorous integrator used to advect these submanifolds. Section \ref{sec:Lorenzresults} illustrates the method in the Lorenz system and illustrates some applications.
The implementation used to obtain all results are found at \cite{ourCode}.

\section{Background: analytic functions, 
Banach algebras of infinite sequences, 
and an a-posteriori theorem} \label{sec:background}

Section \ref{sec:background} reviews some basic properties of analytic 
functions, some standard results from nonlinear analysis,
and establishes some notation used in the remainder of the present work.
This material is standard and is included only for the sake of completeness.
The reader may want to skip ahead to Section \ref{sec:parameterization}, 
and refer back to the present section only as needed.  

\subsection{Analytic functions of several variables, and multi-indexed sequence spaces} \label{sec:sevVar}
Let $d \in \mathbb{N}$ and $z = (z^{(1)}, \ldots, z^{(d)})
\in \mathbb{C}^d$.  We endow $\mathbb{C}^d$ with the 
norm
\[
\|z \| = \max_{1 \leq i \leq d} | z^{(i)}|,
\]
where $|z^{(i)}| = \sqrt{\mbox{real}(z^{(i)})^2 + \mbox{imag}(z^{(i)})^2}$ is 
the usual complex modulus. We refer to the set 
\[
\mathbb{D}^{d} := \left\{ w = (w^{(1)}, \ldots, w^{(d)}) \in 
\mathbb{C}^d : |w^{(i)}| < 1 \mbox{ for all } 1 \leq i \leq d \right\},
\]
as the \textit{unit polydisk} in $\mathbb{C}^d$.  Throughout this paper whenever $d$ is 
understood we write $ \mathbb{D} := \mathbb{D}^d$.
Note that the $d$-dimensional open unit cube $(-1, 1)^d$ is obtained by restricting to the real
part of $\mathbb{D}$.
%In the sequel when we discuss parameterized invariant manifolds and 
%integrate their boundaries we always rescale to work on the
%domain $\mathbb{D}$. The technical details which allow this rescaling are described in detail in Sections \ref{sec:parameterization} and \ref{sec:rigorousIntegration}.

Recall that a function  $f \colon \mathbb{D} \to \mathbb{C}$ 
is {\em analytic} (in the sense of several complex variables)
if for each $z = (z^{(1)}, \ldots, z^{(d)})\in \mathbb{D}$ and $1 \leq i \leq d$,
the complex partial derivative, $\partial f /\partial z^{(i)}$, 
exists and is finite. Equivalently, $f$ is analytic 
(in the sense of several complex variables) if it is analytic (in the usual sense)
in each variable $z^{(i)} \in \mathbb{C}$ 
with the other variables fixed, for  $1 \leq i \leq d$. Denote by 
\[
\| f\|_{C^0( \mathbb{D}, \mathbb{C})} := \sup_{w \in  \mathbb{D}} |f(w^{(1)}, \ldots, w^{(d)})|,
\]
the supremum norm on $ \mathbb{D}$ which we often abbreviate to  $\| f\|_{\infty} := \| f\|_{C^0( \mathbb{D}, \mathbb{C})}$, and let $C^\omega( \mathbb{D})$ denote the set of bounded analytic functions on $ \mathbb{D}$. Recall that if $\{f_n\}_{n=0}^\infty \subset C^\omega(\mathbb{D})$ 
is a sequence of analytic functions and 
\[
\lim_{n\to \infty} \| f - f_n \|_{\infty} = 0,
\]
then $f$ is analytic (i.e.\ $C^\omega( \mathbb{D})$ is a Banach 
space when endowed with the $\| \cdot \|_{\infty}$ norm).
In fact, $C^\omega( \mathbb{D})$
is a Banach algebra, called the \textit{disk algebra}, 
when endowed with pointwise multiplication of functions.  
 
We write $\alpha = (\alpha_1, \ldots, \alpha_d) \in \mathbb{N}^d$ 
for a $d$-dimensional multi-index, where $|\alpha| := \alpha_1 + \ldots + \alpha_d$ is the {\em order}
of the multi-index, and 
$z^{\alpha} := (z^{(1)})^{\alpha_1}\ldots (z^{(d)})^{\alpha_d}$
to denote $z \in \mathbb{C}^d$ raised to the $\alpha$-power.
Recall that a function, $f \in C^\omega(\mathbb{D})$ if and only if for each $z \in  \mathbb{D}$,
$f$ has a power series expansion
\[
f(w) = \sum_{\alpha \in \nn^d} a_\alpha (w - z)^\alpha,
\]
converging absolutely and uniformly in some open neighborhood $U$ 
with $z \in U \subset \mathbb{D}$. For the remainder of this work, we are concerned only with Taylor expansions centered at the origin (i.e.\ $z = 0$ and $U = \mathbb{D}$). Recall that the power series coefficients (or \textit{Taylor coefficients}) 
are determined by certain Cauchy integrals.  
More precisely, for any $f \in C^\omega(\mathbb{D})$ and for any $0 < r < 1$ the $\alpha$-th Taylor coefficient of $f$ centered at $0$ is given explicitly by  
\[
a_\alpha := \frac{1}{(2 \pi i)^d} \int_{|z^{(1)}| = r} 
\ldots \int_{|z^{(d)}| = r}
\frac{f(z^{(1)}, \ldots, z^{(d)})}{
(z^{(1)})^{\alpha_1+1} \ldots (z^{(d)})^{\alpha_d+1}}
\, d z^{(1)} \ldots d z^{(d)},
\]
where the circles $|z^{(i)}| = r$, $1 \leq i \leq d$ are parameterized
with positive orientation.

The collection of all functions 
whose power series expansion centered at the origin converges absolutely and uniformly 
on all of $\mathbb{D}$ is denoted by $\mathcal{B}_d \subset C^\omega(\mathbb{D})$. Let $\mathcal{S}_d$ denote the set of all $d$-dimensional 
multi-indexed sequences of complex numbers. For $a = \{a_{\alpha}\} \in \mathcal{S}_d$ define the norm
\[
\| a \|_{1, d} := \sum_{\alpha \in \nn^d} | a_{\alpha}|,
\]
and let 
\[
\ell^1_d := \left\{a \in \mathcal{S}_d : \| a \|_{1, d} < \infty \right\},
\]
(i.e.\ $\ell^1_d$ is the Banach space of all absolutely summable 
$d$-dimensional multi-indexed sequences of complex numbers).
When $d$ is understood we often abbreviate to $\ell^1$ and $\| a \|_1$. For any $f \in C^\omega(\mathbb{D})$ with Taylor series centered at the origin given by 
\[
f(z) = \sum_{\alpha \in \nn^d} a_\alpha z^\alpha,
\]
let $\mathcal{T}$ denote the mapping given by 
\[
f \stackrel{\mathcal{T}}{\longmapsto} \left\{ a_\alpha \right\}_{\alpha \in \mathbb{N}^d},
\]
which associates an analytic function, $f \in C^\omega(\mathbb{D})$, with the sequence of Taylor coefficients $\{ a_\alpha\}_{\alpha \in \mathbb{N}^d}$ for its power series expansion at $z = 0$. We refer to $\mathcal{T}$ as the \textit{Taylor transform} of $f$ and note that $\mathcal{T}$ is both linear, one-to-one, and takes values in $\mathcal{S}_d$. Moreover, we have the trivial bound
\[
\| f \|_{\infty} \leq \| \TT{f}\|_1,
\]
for each $f \in C^\omega(\mathbb{D})$. Now, let $\mathcal{B}_d^1$ denote 
the collection of all functions $f \in C^\omega(\mathbb{D})$
whose Taylor coefficients are in $\ell^1$ and note that we have the inclusions
\[
\mathcal{B}_d^1 \subset \mathcal{B}_d \subset C^\omega(\mathbb{D}).
\]
In particular, if $a = \{a_\alpha\}_{\alpha \in \mathbb{N}^d} \in \ell^1$, then $a$ 
defines a unique analytic function, $\itt{a} = f \in C^\omega(\mathbb{D})$ given by
\[
f(z) = \sum_{\alpha \in \nn^d} a_\alpha z^\alpha.
\]
We remark that if $f \in \mathcal{B}_d^1$ then $f$ extends uniquely 
to a continuous function on $\overline{\mathbb{D}}$, as the power series coefficients
are absolutely summable at the boundary.  So if 
$f \in \mathcal{B}_d^1$ then
$f \colon \overline{\mathbb{D}} \to \mathbb{C}$
is well defined, continuous on $\overline{\mathbb{D}}$, and analytic on $\mathbb{D}$.

Finally, recall that $\ell^1$ inherits a Banach algebra structure from 
pointwise multiplication, a fact which is 
critical in our nonlinear analysis in Sections \ref{sec:parameterization} and \ref{sec:rigorousIntegration}. 
Begin by defining a total order on $\nn^d$ by setting 
$\kappa \prec \alpha$ if $\kappa_i \leq \alpha_i$ for every $i \in \{1,\dots,d\}$ and 
$\kappa \succ \alpha$ if $\kappa \not\prec \alpha$ (i.e.\ we endow 
$\nn^d$ with the \textit{lexicographic order}). Given $a, b \in \ell^1$, define the binary operator 
$* \colon \ell^1 \times \ell^1 \to \mathbb{S}_d$ by  
\[
[a*b]_\alpha = \sum_{\kappa \prec \alpha} a_\kappa \cdot b_{\alpha-\kappa}.
\]
We refer to $*$ as the \textit{Cauchy product}, and note the following properties:
\begin{itemize}
\item For all $a, b \in \ell^1$ we have
\[
\| a * b \|_1 \leq \| a \|_1 \|b \|_1.
\]
In particular, $\ell^1$ is a Banach algebra when endowed with the 
Cauchy product.
\item Let $f, g \in C^\omega(\mathbb{D})$,
and suppose that 
\[
f(z) = \sum_{\alpha \in \nn^d} a_\alpha z^\alpha
\quad \quad \quad \mbox{and} \quad \quad \quad 
g(z) = \sum_{\alpha \in \nn^d} b_\alpha z^\alpha.
\]
Then $f \cdot g \in C^\omega(\mathbb{D})$ and 
\[
(f \cdot g)(z) = \sum_{\alpha \in \nn^d} [a * b]_\alpha z^\alpha.
\]
In other words, pointwise multiplication of analytic functions corresponds to the Cauchy product in sequence space.  
\end{itemize}

\begin{remark}[Real analytic functions in $\mathcal{B}_d^1$]
If $f \in \mathcal{B}_d^1$ and the Taylor coefficients of $f$ are real, 
then $f$ is real analytic on $(-1, 1)^d$ and continuous 
on $[-1,1]^d$.
\end{remark}

\begin{remark}[Distinguishing space and time]
In Section \ref{sec:rigorousIntegration} it is advantageous both numerically and conceptually to distinguish time from spatial variables. When we need this distinction we write  
$\{a_{m, \alpha}\}_{(m,\alpha) \in \mathbb{N} \times \mathbb{N}^d} = a \in \ell_{d+1}^1$ with the appropriate norm given by  
\[
\| a \|_{1,d+1} = \sum_{m = 0}^\infty \sum_{\alpha \in \nn^d} |a_{m, \alpha}|.
\]
In this setting, $a$ defines a unique analytic function $\itt{a} = f \in C^\omega(\mathbb{D}^{d+1})$ given by  
\[
f(z,t) = \sum_{m=0}^\infty \sum_{\alpha \in \nn^d} a_{m, \alpha} z^\alpha t^m,
\]
where $z$ is distinguished as the (complex) space variable and $t$ is the time variable. Analogously, we extend the ordering on multi-indices to this distinguished case by setting $(j,\kappa) \prec (m,\alpha)$ if $j \leq m$ and $\kappa \prec \alpha$ as well as the Cauchy product by 
\[
[a*b]_{m,\alpha} = 
\sum_{j \leq m} \sum_{\kappa \prec \alpha} a_{j,\kappa} \cdot b_{m-j,\alpha-\kappa}.
\]
\end{remark}

\subsection{Banach spaces and linear algebra} 
\label{sec:linearAlgebra}
The validation methods utilized in this work are based on a set of principles for obtaining mathematically rigorous solutions to nonlinear operator equations with computer assistance referred to as the {\em radii polynomial approach}. A key feature of this philosophy is the characterization of a nonlinear problem in the space of analytic functions as a zero finding problem in sequence space. Specifically, our methods will seek a (Fr\'{e}chet) differentiable map in $\ell^1$ and require (approximate) computation of this map and its derivative. 

%This necessitates discussion of linear functionals on sequence space. 

%To begin a discussion of linear functionals on sequence space, let $b \in \mathcal{S}_d$ and define the norm 
%\[
%\| b \|_{\infty, d} := \sup_{|\alpha| \geq 0} | b_\alpha |,
%\]
%and the space
%\[
%\ell_d^\infty := \left\{ b \in \mathcal{S}_d : \| b\|_{\infty,d} < \infty \right\},
%\]
%which we often abbreviate to $\ell^\infty$ and $\| b \|_\infty$ (as long as there
%is no possibility of confusion between $\| f\|_\infty$ and $\| b \|_\infty$).  
%Recall that $\ell^\infty$ is a Banach space which is isometrically isomorphic 
%to the Banach dual of $\ell^1$. Specifically, if $a \in \ell^1$ and $b \in \ell^\infty$ the isomorphism is given explicitly by
%\[
%l_b(a) := \sum_{\alpha \in \nn^d} a_\alpha b_\alpha.
%\]
%We also note that this leads to the useful bound
%\[
%\left| \sum_{\alpha \in \nn^d} a_\alpha b_\alpha \right|  \leq  \| b \|_\infty \| a \|_1.
%\]

For our purposes, we are interested in bounded linear operators defined on $\ell^1$. Let $\mathcal{L}(\ell^1, \ell^1)$ denote the vector space of bounded linear operators from $\ell^1$ to itself, which we shorten to $\mathcal{L}(\ell^1)$, equipped with the operator norm induced by $\norm{\cdot}_{1}$. For this discussion we utilize the notation with space/time distinguished. To avoid confusion over indices, we denote indices for linear operators {\em inside} square brackets and components of vectors {\em outside} square brackets. Now, we fix a basis for $\ell^1$ composed of 
$\{e^{j \kappa}\}$  where
\[ 
[e^{j \kappa}]_{m,\alpha} =
\left(
\begin{array}{cc}
1 & (j,\kappa) = (m,\alpha) \\
0 & \text{otherwise}
\end{array}
\right),
\]
and we specify an element $A \in \mathcal{L}(\ell^1)$, 
by its action on these basis vectors which we denote by 
\[
A^{j k} =  A\cdot e^{j k}
\]
%\[
%A^{m \alpha} =  A\cdot e^{m \alpha}
%\]
With this notation in place, our first goal is to compute a formula for the operator norm on $\mathcal{L}(\ell^1)$ defined by 
\[
\norm{A}_{1} = 
\sup\limits_{\norm{h}=1} \norm{A\cdot h}_{1}.
\]
\begin{prop}
	For $A \in \mathcal{L}(\ell^1)$, the operator norm is given by
	\[
	\norm{A}_{1} = 
	\sup\limits_{(j,\kappa) \in \nn \times \nn^d}  \norm{A^{j \kappa}}_{1}
	\]
\end{prop}
\begin{proof}
We define $C = 	\sup\limits_{(j,\kappa) \in \nn \times \nn^d}  \norm{A^{j \kappa}}_{1}$ which 
is finite since $A$ is a bounded linear operator. Suppose $h \in \ell^1$ is a unit vector which we express in the above basis as
\[
h = \sum_{j=0}^{\infty} \sum_{\kappa \in \nn^d} h_{j,k}e^{j \kappa}.
\] 
Then for each $(m,\alpha) \in \nn \times \nn^d$ we have 
\[
\left|[A \cdot h]_{m,\alpha}\right| = \left|\sum_{j=0}^{\infty} \sum_{\kappa \in \nn^d} [A^{j \kappa}]_{m,\alpha} \cdot h_{j,\kappa}\right|. 
\]
Applying this directly for each coordinate in $A\cdot h$ leads to the following estimate
\begin{align*}
\norm{A \cdot h}_{1} & = \sum_{m=0}^{\infty} \sum_{\alpha \in \nn^d}  \left|\sum_{j=0}^{\infty} \sum_{\kappa \in \nn^d} [A^{j \kappa}]_{m,\alpha} \cdot h_{j,\kappa}\right| \\
& \leq \sum_{m=0}^{\infty} \sum_{\alpha \in \nn^d}  
	\sum_{j=0}^{\infty} \sum_{\kappa \in \nn^d} \left| [A^{j \kappa}]_{m,\alpha} \right| \cdot \left| h_{j,\kappa}\right| \\
& \leq \sum_{j=0}^{\infty} \sum_{\kappa \in \nn^d} \left| h_{j,\kappa}\right|
	\sum_{m=0}^{\infty} \sum_{\alpha \in \nn^d} \left| [A^{j \kappa}]_{m,\alpha} \right| \\
& \leq \sum_{j=0}^{\infty} \sum_{\kappa \in \nn^d} \left| h_{j,\kappa}\right|
\norm{A^{j \kappa}}_{1}\\	
& \leq C \sum_{j=0}^{\infty} \sum_{\kappa \in \nn^d} \left| h_{j,\kappa}\right| \\
& = C
\end{align*}
and taking the supremum over all unit vectors in $\ell^1$ we have $\norm{A}_1 \leq C$. Conversely, for any $\epsilon > 0$ we may choose $(j,\kappa) \in \nn \times \nn^d$ such that $\norm{A^{j\kappa}}_1 > C - \epsilon$. It follows that 
\[
\norm{A}_1 \geq \norm{A \cdot e^{j,\kappa}}_1 > C - \epsilon
\]
and we conclude that $\norm{A}_1 \geq C$. 
\end{proof}

Next, we define specific linear operators which play an important role in the developments to follow. The first 
operator is the multiplication operator induced by an element in $\ell^1$. Specifically, 
for a fixed vector, $a \in \ell^1$, there exists a unique linear operator, $T_a$, whose action is given by 
\begin{equation}
\label{deftimesoperator}
T_a\cdot u = a*u 
\end{equation}
for every $u \in \ell^1$. 
With respect to the above basis we can write $T_a \cdot e^{j \kappa}$ explicitly as 
\[
[T_a^{j \kappa}]_{m,\alpha} = 
\left(
\begin{array}{cc}
a_{j-m,\kappa - \alpha} & (m,\alpha) \prec (j,\kappa) \\
0 & \text{otherwise}
\end{array}
\right)
\]
which can be verified by a direct computation. The second operator is a coefficient shift followed by padding with zeros, which we will denote by $\eta$. Its action on $u  \in \ell^1$ is given explicitly by
\begin{equation}
\label{defeta}
[\eta \cdot  u]_{m,\alpha} =
\left\{
\begin{array}{cc}
0 & \text{if }m=0  \\
u_{m-1,\alpha} & \text{if } m \geq 1 \\
\end{array}
\right.
\end{equation}
Additionally, we introduce the ``derivative'' operator whose action on vectors will be denoted by $'$. Its action on $u \in \ell^1$ is given by the formula
\begin{equation}
\label{defdiff}
[u']_{m,\alpha} =
\left\{
\begin{array}{cc}
u_{m,\alpha} & \text{if }m=0  \\
mu_{m,\alpha} & \text{if } m \geq 1 \\
\end{array}
\right.
\end{equation}
The usefulness in these definitions is made clear in Section \ref{sec:rigorousIntegration}. 

Finally, we introduce several properties of these operators which allow us to estimate their norms. The first is a generalization of the usual notion of a lower-triangular matrix to higher order tensors. 
\begin{prop}
We say an operator, $A \in \mathcal{L}(\ell^1)$, is {\em upper triangular} with respect to $\{e^{j \kappa}\}_{(j,\kappa) \in \nn \times \nn^d} $ if $A^{m \alpha} \in \operatorname{span} \{e^{j \kappa}: (j,\kappa) \prec (m,\alpha)\}$ for every $(m,\alpha) \in \nn \times \nn^d$. Then, each of the operators defined above is upper triangular. The proof for each operator follows immediately from their definitions. 
\end{prop}

Next, we introduce notation for decomposing a vector $u \in \ell^1$ into its finite and infinite parts. 
Specifically, for fixed $(m,\alpha) \in \nn \times \nn^{d}$ we denote the finite truncation of $u \in \ell^1$ to $(m, \alpha)$-many terms (embedded in $\ell^1$) by 
\begin{equation}
\label{eq:truncDef}
u^{m\alpha} = 
\left\{
\begin{array}{cc}
u_{j,\kappa} & (j,\kappa) \prec (m,\alpha) \\
0 & \text{otherwise}
\end{array}
\right. ,
\end{equation}
and we define the infinite part of $u$ by $u^{\infty} = u - u^{m \alpha}$. From the point of view of Taylor series, $u^{m \alpha}$ are the coefficients of a polynomial approximation obtained by truncating $u$ to $m$ temporal terms and $\alpha_i$ spatial terms in the $i^{th}$ direction, and $u^{\infty}$ represents the tail of the Taylor series. With this notation we establish several useful estimates for computing norms in $\ell^1$. 
\begin{prop}
\label{prop:operatorbounds}
Fix $a \in \ell^1$ and suppose $u \in \ell^1$ is arbitrary. Then the following estimates hold for all $(m,\alpha) \in \nn \times \nn^d$.
\begin{eqnarray}
\norm{T_a \cdot u}_{1} \leq & \norm{a}_{1}  \norm{u}_{1} \\ 
\norm{\eta(u)}_{1} = &\norm{u}_{1} %\\
%\norm{u^{\infty}}_{1}\leq & \frac{1}{m} \norm{u}_{1}
\end{eqnarray}
\end{prop}
\noindent The proof is a straightforward computation. 

\subsection{Product spaces}
In the preceding discussion we considered the vector space structure on $\ell^1$ and described linear operators on this structure. In this section, we recall that $\ell^1$ is an algebra, and therefore it is meaningful to consider vector spaces over $\ell^1$ where we consider elements of $\ell^1$ as ``scalars''. Indeed, an $n$-dimensional vector space of this form is the appropriate space to seek solutions to the invariance equation described in Section \ref{sec:parameterization} as well as IVPs which we describe in Section \ref{sec:rigorousIntegration}. To make this more precise we define
\begin{equation}
\label{defba1}
\mathcal{X} =  \left\{
\{u_{m,\alpha}^{(i)}\} \subset \cc^{d}: \sum\limits_{m = 0}^{\infty} \sum\limits_{\alpha \in \nn^d} |u_{m,\alpha}^{(i)}| < \infty \ \text{for all } 1\leq i\leq n
 \right\},
\end{equation}
and we recognize that an element $u\in \mathcal{X}$ defines a unique analytic function in $d$-many variables, taking values in $\cc^n$. We recall that the restriction of this function to a single coordinate defines a scalar analytic function with coefficients in $\ell^1$. Thus, $\mathcal{X}$ can be equivalently defined as an $n$-dimensional vector space over $\mathcal{X}$ given by
\begin{equation}
\label{defba2}
\mathcal{X} = \underbrace{\ell^1 \times \ell^1 \times \dots \ell^1 }_{n \text{-copies}} = (\ell^1)^{n},
\end{equation}
and a typical element $u \in \mathcal{X}$ takes the form $u = (u^{(1)},\dots,u^{(n)})$ with each $u^{(i)} \in \ell^1$. When solving nonlinear problems in $\mathcal{X}$, we will typically adopt the notation and point of view in Equation~\eqref{defba2}. Next, we equip $\mathcal{X}$ with the norm given by
\begin{equation}
\label{banorm}
\norm{u}_\mathcal{X} = \max_{1 \leq i \leq n} \{ ||u^{(i)}||_{1}\}.
\end{equation}
Finally, define multiplication in $\mathcal{X}$ componentwise. Specifically, if $u,v \in \mathcal{X}$, then each is an $n$-length vector of scalars from $\ell^1$ and the multiplication defined by
\begin{equation}
\label{batimes}
[u*v]_{m,\alpha} = ([u^{(1)}*v^{(1)}]_{m,\alpha}, \dots, [u^{(n)}*v^{(n)}]_{m,\alpha})
\end{equation}
makes $\mathcal{X}$ into a Banach algebra. The decomposition of $u \in \mathcal{X}$ into a finite projection and infinite tail is also defined componentwise. 

Let $\mathcal{L}(\mathcal{X})$ denote the vector space of linear operators on $\mathcal{X}$, and suppose $A \in \mathcal{L}(\mathcal{X})$. Since $\mathcal{X}$ is a finite dimensional vector space over $\ell^1$, it follows that for any fixed basis of $\mathcal{X}$ over $\ell^1$, we can identify $A$ with some $n \times n$ square matrix, $Q$, so that the action of $A$ on a vector $u \in \mathcal{X}$ is left multiplication by $Q$ which has the form
\[
Q = 
\left(
\begin{array}{cccc}
q_{(11)} & q_{(12)} & \dots & q_{(1n)}\\
q_{(21)} & q_{(22)} & \dots & q_{(2n)}\\
\vdots & \vdots & \ddots & \vdots\\
q_{(n1)} & q_{(n2)} & \dots & q_{(nn)}\\ 
\end{array}
\right)
\]
It is a standard result that 
the sup norm defined in Equation \eqref{banorm}
induces the operator norm given by 
\[
\norm{A}_{\mathcal{X}} = \max\limits_{1 \leq i \leq n} 
\left\{r_i : \ r_i = \sum_{j = 1}^{n} \norm{q_{(ij)}}_{1} 
\right\}.
\]
where $q_{(ij)}$ are bounded linear operators and we recall that $\norm{q_{(ij)}}_{1}$ denotes their operator norm.

\subsection{A-posteriori analysis for nonlinear operators between Banach spaces} \label{sec:aPos}
The discussion in Section \ref{sec:sevVar} motivates the approach to 
validated numerics/computer assisted proof
adopted below.  Let $d, n \in \mathbb{N}$ and consider a nonlinear operator 
$\Psi \colon C^\omega(\mathbb{D}^d)^n \to C^\omega(\mathbb{D}^d)^n$
(possibly with $\Psi$ only densely defined). 
Suppose that we want to solve the equation
\[
\Psi(f) = 0.
\] 
Projecting the $n$ components of $\Psi$ into sequence space results in an equivalent map $F \colon \left(\mathcal{S}_d\right)^n \to \left(\mathcal{S}_d\right)^n$ on the coefficient level. The transformed problem is truncated by simply restricting our attention to Taylor coefficients with order $0 \leq |\alpha| \leq N$ for some $N \in \mathbb{N}$. We denote by $F^N$ the truncated map. The problem $F^N = 0$ is now solved using any convenient numerical method, and we denote by $a^N$ the appropriate numerical solution, and by $\overbar{a} \in \mathcal{X}$ the infinite sequence which results from extending $a^N$ by zeros. 

We would like now, if possible, to prove that there is an $\tilde{a} \in \mathcal{X}$ near $\overbar{a}$, which satisfies $F(\tilde{a}) = 0$. 
Should we succeed, then by the discussion in Section \ref{sec:sevVar}, the function $f = (f_1, \ldots, f_n) \in \left(C^\omega(\mathbb{D}^d)\right)^n$ with Taylor coefficients given by $a$ is a zero of $\Psi$ as desired. The following proposition, which is formulated in general for maps between Banach spaces, provides a framework for implementing such arguments.

\begin{prop} \label{prop:newton}
Let $\mathcal{X}$, $\mathcal{Y}$ be Banach spaces and $F \colon \mathcal{X} \to \mathcal{Y}$
be a Fr\'{e}chet differentiable mapping.  
Fix $\overbar{a} \in \mathcal{X}$ and suppose there are bounded linear operators 
$A^{\dagger} \in \mathcal{L}(\mathcal{X}, \mathcal{Y})$,
$A \in \mathcal{L}(\mathcal{Y}, \mathcal{X})$, with $A$ one-to-one.
Assume that there are non-negative constants, $r,Y_0,Z_0,Z_1,Z_2$, 
satisfying the following bounds for all $x \in \overbar{B_r(\overbar{a})}$: 
\begin{eqnarray}
\label{Y0}
%\normX{AF(\overbar{a})} \leq Y_0 \\
\normX{AF(\overbar{a})} \leq Y_0 \\
\label{Z0}
\normX{\text{Id} - AA^{\dagger}} \leq Z_0 \\
\label{Z1}
%\normX{A(A^{\dagger}-DF(\overbar{a}))} \leq Z_1 \\
\normX{A(A^{\dagger}-DF(\overbar{a}))} \leq Z_1 \\
\label{Z2}
\normX{A(DF(x) - DF(\overbar{a}))} \leq Z_2 \normX{x-\overbar{a}} \\
\label{r}
Y_0 + (Z_0 + Z_1)r + Z_2r^2 < r.
\end{eqnarray}
Then there exists a unique $\tilde a \in B_{r}(\overbar{a})$
so that $F(\tilde a) = 0$. 
%From our above discussion we observe that this fixed point must be $a$ and it follows that $r$ is an explicit bound on the 
%approximation error in the $\ell^1$-topology.   
\end{prop}

\begin{proof}
Consider the nonlinear operator 
$T \colon \mathcal{X} \to \mathcal{X}$ defined by 
\[
T(x) = x - A F(x).
\]
Since $A$ is one-to-one, $\tilde a \in \mathcal{X}$ is a zero of $F$ if and only if $\tilde a$ is a 
fixed point of $T$.  The idea of the proof is to use the Banach fixed point theorem 
to establish the existence of a unique fixed point in $B_{r}(\overbar{a})$.

Let $\text{Id}$ denote the identity map on $\mathcal{X}$, suppose $x \in \overbar{B_r(\overbar{a})}$,
and note that $DT(x) = \text{Id} - A DF(x)$.
Then
\begin{align*}
\normX{DT(x)} & = \normX{\text{Id} - ADF(x)}\\
	& = \normX{(\text{Id}-AA^{\dagger}) + A(A^{\dagger}-DF(\overbar{a})) + A(DF(\overbar{a}) - DF(x))} \\
	& \leq  \normX{\text{Id}-AA^{\dagger}} + \normX{A(A^{\dagger}-DF(\overbar{a}))} + \normX{A(DF(\overbar{a}) - DF(x))}.
\end{align*}
Taking this together with assumptions~\eqref{Z0}, \eqref{Z1}, and \eqref{Z2} we obtain the bound
\begin{equation}
\label{DT bound}
\sup_{x \in \overbar{B_r(\overbar{a})}} \normX{DT(x)} \leq Z_0 + Z_1 + Z_2r.
\end{equation}
Now, if $x \in \overline{B_r(\overbar{a})}$, 
then applying the bound~\eqref{Y0} and invoking the Mean Value Theorem yields the estimate
\begin{align}
\normX{T(x) - \overbar{a}} & \leq \normX{T(x) - T({\overbar{a}})} + \normX{T(\overbar{a}) - \overbar{a}} \nonumber \\
& \leq \sup_{x \in B_r(\overbar{a})} \normX{DT(x)}\cdot \normX{x - \overbar{a}} + \normX{AF(\overbar{a})} \nonumber \\
& \leq Y_0 + (Z_0 + Z_1)r + Z_2r^2 \nonumber \\
& < r \label{eq:proofStrict}
\end{align}
where the last inequality is due to Equation~\eqref{r}. This proves that $T$ maps $\overline{B_r(\overbar{a})}$
into itself.  In fact, $T$ sends $\overline{B_r(\overbar{a})}$ into 
$B_r(\overbar{a})$, by the strict inequality.

Finally, assume $x,y \in \overline{B_r(\overbar{a})}$ 
and apply the bound of Equation~\eqref{DT bound} with the Mean Value Theorem 
once more to obtain the contraction estimate 
\begin{align}
\normX{T(x) - T(y)} & \leq \sup_{x \in \overline{B_r(\overbar{a})}} \norm{DT(x)}_{\mathcal{X}} \cdot \norm{x - y}_{\mathcal{X}} \nonumber \\
& \leq (Z_0 + Z_1 + Z_2r) \norm{x-y}_{\mathcal{X}} \nonumber \\
& < \left(1 - \frac{Y_0}{r} \right)\norm{x-y}_{\mathcal{X}} 
\end{align}
where the second to last line follows from another application of Equation~\eqref{r} 
and the last line from noticing that $0 < {Y_0}/{r} < 1$. 
Since $1 - Y_0/r < 1$, the Contraction Mapping Theorem is satisfied 
on $\overline{B_r(\overbar{a})}$.  By the strict inequality of Equation \eqref{eq:proofStrict}
we conclude that $T$ has a unique fixed point  
$\tilde a \in B_r(\overbar{a})$ and  it follows that $\tilde a$ is the unique zero of $F$
in $\in B_r(\overbar{a})$.
\end{proof}

\begin{remark}
A few remarks on the intuition behind the terms appearing in the proposition are in order.
Intuitively speaking, $p(r) < 0$ occurs when $Y_0$, $Z_0$, $Z_1$ are small, and 
$Z_2$ is not too large.  Here $Y_0$ measures the defect associated with 
$\overbar{a}$ (i.e.\ $Y_0$ small means that we have a ``close'' approximate solution).  
We think of $A^\dagger$ as an approximation of the differential $DF(\overbar{a})$, 
and $A$ as an approximate inverse of $A^\dagger$.  Then $Z_0, Z_1$ measure the 
quality of these approximations. These approximations are used as it is 
typically not possible to invert $DF(\overbar{a})$ exactly.
Finally $Z_2$ is in some sense a measure
of the local ``stiffness'' of the problem.  For example $Z_2$ is often taken 
as any uniform bound on the second derivative of $F$ near $\overbar{a}$.
The choice of the operators $A, A^\dagger$ is problem dependent and best illustrated 
through examples.   Finally we remark that it is often unnecessary 
to specify explicitly the space $\mathcal{Y}$.  Rather, what is important is that 
for all $x \in \mathcal{X}$ we have that 
$A F(x) \in \mathcal{X}$ and that $A A^\dagger x \in \mathcal{X}$.
\end{remark}

%\subsection{Radii polynomials} \label{sec:radPoly}
\begin{remark} \label{rmk:radPoly}

Following
\cite{MR2338393, MR2443030, MR2630003, MR2776917}, we
exploit the radii polynomial method to organize the computer assisted
argument giving validated error bounds for our integrator. 
In short, this amounts to rewriting the contraction mapping condition above by defining the radii polynomial
\[
p(r) = Z_2r^2 + (Z_0 + Z_1 - 1)r + Y_0
\]
and noting that  the hypothesis of Proposition~\ref{prop:newton} in equation~\eqref{r} is satisfied for any $r> 0$ such that $p(r) < 0$. 
%The Lipschitz bound, $Z_2$, is positive and varies continuously as a function of $r$. 
It follows that the minimum root of $p$ (if it exists) gives a sharp bound on the error, and if $p$ has distinct roots, $\{r_-,r_+\}$, then $p < 0$ on the 
entire interval $(r_-,r_+)$. The isolation bound $r_+$ is theoretically infinite, as the solutions of 
initial value problems are globally unique.  However the width of the interval $r_+ - r_-$ provides 
a quantitative measure of the difficulty of a given proof, as when this difference is zero the proof fails.
\end{remark}

\section{The parameterization method for (un)stable manifolds} \label{sec:parameterization}
The parameterization method is a general functional analytic framework for analyzing 
invariant manifolds, based on the idea of studying dynamical conjugacy 
relationships.  The method is first developed in a series of papers
\cite{MR1976079, MR1976080, MR2177465, MR2299977, MR2240743, MR2289544}.
By now there is a small but thriving community of researchers applying and 
extending these ideas, and a serious review of the literature would take us 
far afield.  Instead we refer the interested reader to the recent book 
\cite{mamotreto}, and turn to the task of reviewing as much of the method
as we use in the present work.

Consider a real analytic vector field $f \colon \mathbb{R}^n \to \mathbb{R}^n$, 
with $f$ generating a flow $\Phi \colon U \times \mathbb{R} \to \mathbb{R}^n$,
for some open set $U \subset \mathbb{R}^n$.
Suppose that $p \in U$ is an equilibrium solution, and let
$\lambda_1, \ldots, \lambda_d \in \mathbb{C}$ denoted the stable 
eigenvalues of the matrix $Df(p)$.  Let $\xi_1, \ldots, \xi_d \in \mathbb{C}^n$ 
denote a choice of
associated eigenvectors. In this section we write $B = B_1^d = \left\{ s \in \mathbb{R}^d \, : \,  \| s \| < 1 \right\}$,
for the unit ball in $\mathbb{R}^d$.  

The goal of the Parameterization Method is to solve the \textit{invariance equation}
\begin{equation}\label{eq:invEqnFlows1}
f(P(s)) = \lambda_1 s_1 \frac{\partial}{\partial s_1} P(s) + 
\ldots + \lambda_d s_d \frac{\partial}{\partial s_d} P(s),
\end{equation}
on $B$, 
subject to the first order constraints
\begin{equation} \label{eq:firstOrder1}
P(0) = p
\quad \quad \quad \quad \mbox{and} \quad \quad \quad \quad 
\frac{\partial}{\partial s_j} P(0) = \xi_j
\end{equation}
for $1 \leq j \leq d$.
From a geometric point of view, Equation \eqref{eq:invEqnFlows1}
says that the push forward by $P$ of the linear vector field generated by the 
stable eigenvalues is equal to the vector field $f$ restricted to the 
image of $P$.  In other words Equation \eqref{eq:invEqnFlows1}
provides an infinitesimal conjugacy between the stable linear
dynamics and the nonlinear flow, but only on the manifold parameterized
by $P$.  More precisely we have the  following Lemma.

\begin{figure}[t!]
\center{
  \includegraphics[width=0.6\linewidth]{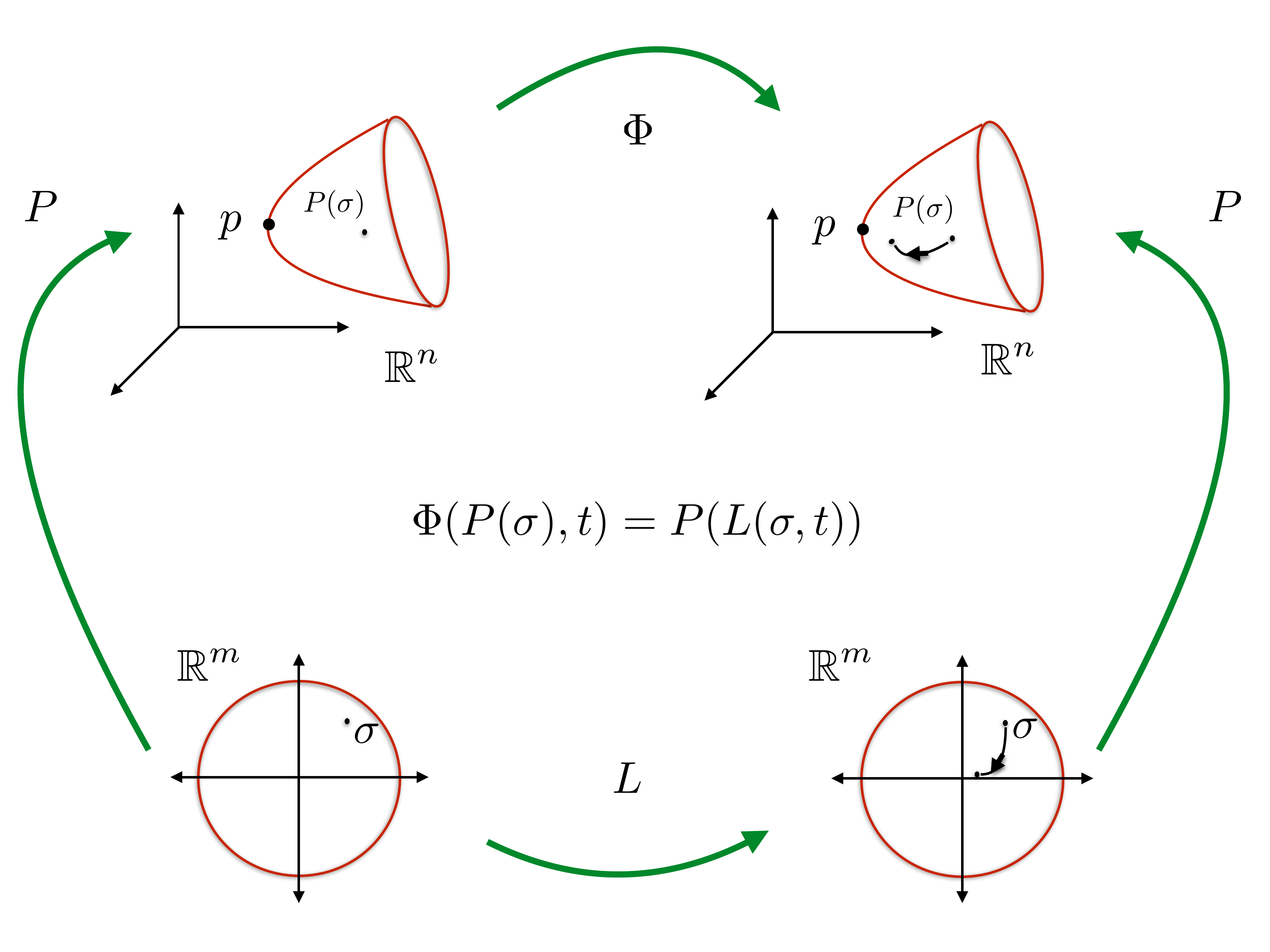}
}
\caption{Illustration of the flow conjugacy: the commuting diagram 
explains the geometric content of Equation \eqref{eq:conjEq1},
and explains the main property we want the parameterization $P$
to have.  Namely, we want that applying the linear flow $L$
in parameter space for a time $t$ and then lifting to the image of $P$ is 
the same as first lifting to the image of $P$ and then applying the 
non-linear flow $\Phi$ for time $t$.} \label{fig:conjugacyFP}
\end{figure}

\begin{lemma}[Parameterization Lemma] \label{lem:parmLemma1} 
{\em
Let $L \colon \mathbb{R}^d \times \mathbb{R} \to \mathbb{R}^d$
be the linear flow
\[
L(s,t) = \left(
e^{\lambda_1 t} s_1, \ldots, e^{\lambda_d t} s_d
\right).
\]
Let  $P \colon B \subset \mathbb{R}^d \to \mathbb{R}^n$
be a smooth function satisfying Equation \eqref{eq:invEqnFlows1}
on $B$ and subject to the constraints given by Equation \eqref{eq:firstOrder1}.
Then $P(s)$ satisfies the flow conjugacy 
\begin{equation} \label{eq:conjEq1}
\Phi(P(s), t) = P(L(s,t)),
\end{equation}
for all $t \geq 0$ and $s \in  B$.
}
\end{lemma}

\medskip

For a proof of the Lemma and more complete discussion we refer to 
\cite{fastSlow}.  The flow conjugacy described by Equation~\eqref{eq:conjEq1}
is illustrated pictorially in Figure \ref{fig:conjugacyFP}.
Note that $L$ is the flow generated by the vector field
\[
\frac{d}{dt} s_j = \lambda_j s_j, 
\quad \quad \quad \quad 1 \leq j \leq d,
\] 
i.e.\ the diagonal linear system with rates given by the 
stable eigenvalues of $Df(p)$.
Note also that the converse of the lemma holds, 
so that $P$ satisfies the flow conjugacy if and only if $P$ satisfies the 
infinitesimal conjugacy.
We remark also that $P$ is (real) analytic if $f$ is analytic 
\cite{MR1976080, MR2177465}.
 
Now, one checks that if $P$ satisfies the flow conjugacy 
given Equation \eqref{eq:conjEq1}, then
\[
P(B) \subset W^s(p),
\]
i.e.\ the image of $P$ is a local stable manifold.  This is seen 
by considering that
\[
\lim_{t \to \infty} \Phi(P(s), t) = \lim_{t \to \infty} P(L(s,t)) = p
\quad \quad \quad \mbox{for all } s \in  B \subset \mathbb{R}^d,
\]
which exploits the flow conjugacy, 
the fact that $L$ is stable linear flow, and that $P$ is continuous.

It can be shown that solutions of Equation \eqref{eq:invEqnFlows1}
are unique up to the choice of the scalings of the eigenvectors.
Moreover, on the level of the power series representation of the solution,
the scaling of the eigenvectors determines 
the decay rates of the Taylor coefficients of $P$.  
Proofs are found for example in \cite{MR2177465}.
These facts are used to show that, once we fix the domain of the praameterization to $B$,  
the solution $P$ parameterizes a larger or smaller local portion of the stable manifold
depending only on the choice of the eigenvector scalings.  
In practice this freedom in the choice in the scalings of the eigenvectors is exploited 
to stabilize numerical computations.  See for example   
\cite{maximeJPMe}.

The existence question for Equation \eqref{eq:invEqnFlows1} is 
somewhat more subtle.  While the stable manifold theorem guarantees
the existence of stable manifolds for a hyperbolic fixed point, 
Equation \eqref{eq:invEqnFlows1}
provides more.  Namely a chart map which recovers the dynamics
on the invariant manifold via a flow conjugacy relation.
It is not surprising then that some additional assumptions are necessary
in order to guarantee solutions of Equation \eqref{eq:invEqnFlows1}.

The necessary and sufficient conditions are given by considering 
certain \textit{non-resonance conditions} between the stable eigenvalues.
We say that the stable eigenvalues are \textit{resonant} if there exists an
$\alpha = (\alpha_1, \ldots, \alpha_d) \in \mathbb{N}^d$ so that
\begin{equation}\label{eq:defRes1}
\alpha_1 \lambda_1 + \ldots + \alpha_d \lambda_d = \lambda_j
\quad \quad \quad \quad \mbox{for some }  1 \leq j \leq d.
\end{equation}
The eigenvalues are \textit{non-resonant} if the condition 
given in Equation \eqref{eq:defRes1} fails for all $\alpha \in \mathbb{N}^d$.
Note that since $\lambda_j$, $\alpha_j$, $1 \leq j \leq d$ all have the same 
sign, there are only a finite number of opportunities for a resonance. Thus, 
in spite of first appearances, Equation \eqref{eq:defRes1} imposes only 
a finite number of conditions between the stable eigenvalues. 
The following provides necessary and sufficient conditions that some solution 
of Equation \eqref{eq:invEqnFlows1} exists.

\begin{lemma}[A-priori existence] \label{lem:existence1}
Suppose that $\lambda_1, \ldots, \lambda_d$ are non-resonant.  
Then there is an $\epsilon > 0$ such that
\[
\| \xi_j \| \leq \epsilon
\quad \quad \quad \quad \mbox{for each } 1 \leq j \leq d,
\]
implies existence of a solution to Equation \eqref{eq:invEqnFlows1} satisfying the constraints 
given by Equation \eqref{eq:firstOrder1}. 
\end{lemma}

A proof of a substantially more general theorem for densely defined 
vector fields on Banach spaces (which certainly covers the present case)
is found in \cite{parmPDE}.  Other general theorems (for maps on 
Banach spaces) are found in \cite{MR1976079, MR1976080, MR2177465}.
We note that in applications we would 
like to pick the scalings of the eigenvectors as large as possible, in 
order to parameterize as large a portion of the manifold as possible, 
and in this case we have no guarantee of existence.  This is motivates
the a-posteriori theory developed in
\cite{parmPDE, maximePOmanifolds, maximeJPMe}, which we 
utilize in the remainder of the paper.  

Finally, we note that even when the eigenvalues are resonant it is still possible 
to obtain an analogous theory by modifying the map 
$L$.  As remarked above, there can only be finitely many resonances 
between $\lambda_1, \ldots, \lambda_d$.  Then in the resonant case
$L$ can be chosen a polynomial which ``kills'' the resonant terms, i.e.\ 
we conjugate to a polynomial rather than a linear vector field in $\mathbb{R}^d$. 
Resonant cases are treated in detail in \cite{MR1976079, parmChristian}. 
Of course all the discussion above goes through for unstable manifolds
by time reversal, i.e.\ considering the vector field $- f$.

\subsection{Formal series solution of equation \eqref{eq:invEqnFlows1}}
In practical applications our first goal is to solve Equation \eqref{eq:invEqnFlows1}
numerically.  Again, it is shown in \cite{MR2177465} that if $f$ is analytic, then
$P$ is analytic as well.  Based on the discussion of the 
previous section we look for a choice of scalings of the 
eigenvectors and power series coefficients $p_\alpha \in \mathbb{R}^n$
so that 
\begin{equation}
P(s) = \sum_{\alpha \in \nn^d} p_\alpha s^\alpha,
\end{equation}
is the desired solution for $s \in  B$.

Imposing the linear constraints given in Equation \eqref{eq:firstOrder1}
leads to 
\[
p_0 = p
\quad \quad \quad \quad \mbox{and} \quad 
p_{\alpha_j} = \xi_j \mbox{ for } 1 \leq j \leq d.
\]
Here $0$ denotes the zero multi-index in $\mathbb{N}^d$, and 
$\alpha_j$ for $1 \leq j \leq d$ are the first-order multi-indices 
satisfying $|\alpha_j| = 1$. The remaining coefficients are determined by 
power matching. Note that 
\[
\lambda_1s_1 \frac{\partial}{\partial s_1} P(s) + \ldots +
\lambda_d s_d \frac{\partial}{\partial s_d} P(s) 
= \sum_{\alpha \in \nn^d} (\alpha_1 \lambda_1 + \ldots + \alpha_d \lambda_d) p_{\alpha} s^\alpha.
\]
Returning to Equation \eqref{eq:invEqnFlows1} we let 
\[
f[P(s)] = \sum_{\alpha \in \nn^d} q_\alpha s^\alpha, 
\]
so that matching like powers leads to the homological equations
\[
 (\alpha_1 \lambda_1 + \ldots + \alpha_d \lambda_d) p_{\alpha} - q_\alpha = 0,
\]
for all $|\alpha| \geq 2$.  Of course each $q_\alpha$ depends on $p_\alpha$ in 
a nonlinear way, and solution of the homological equations is best illustrated
through examples.  

\bigskip

\noindent \textbf{Example: equilibrium solution of Lorenz with two 
stable directions.}
Consider the Lorenz system defined by the vector field 
$f \colon \mathbb{R}^3 \to \mathbb{R}^3$ where
\begin{equation} \label{eq:LorenzField}
f(x,y,z) = 
\left(
\begin{array}{c}
\sigma (y - x) \\
x(\rho - z) - y \\
xy - \beta z
\end{array}
\right).
\end{equation}
For $\rho > 1$ there are three equilibrium points 
\[
p^0 = 
\left(
\begin{array}{c}
0 \\
0 \\
0
\end{array}
\right), 
\quad \quad \text{and} \quad \quad
p^{\pm} = 
\left(
\begin{array}{c}
\pm \sqrt{\beta(\rho - 1)} \\
\pm \sqrt{\beta(\rho-1)} \\
\rho - 1
\end{array}
\right).
\]
Choose one of the three fixed points above and denote it by 
$ p \in \mathbb{R}^3$. Assume that $Df(p)$ has two eigenvalues $\lambda_1, \lambda_2 \in \mathbb{C}$ of 
the same stability type (either both stable or both unstable) 
and assume that the remaining eigenvalue $\lambda_3$
has opposite stability. In this case we have $d = 2$ and the invariance equation is given by
\begin{equation} \label{eq:lorInvEqn}
\lambda_1 s_1 \frac{\partial}{\partial s_1} P(s_1, s_2)
+ \lambda_2 s_2 \frac{\partial}{\partial s_2} P(s_1, s_2) 
= f[P(s_1, s_2)],
\end{equation}
and we look for its solution in the form
\[
P(s_1, s_2) 
= \sum_{\alpha \in \nn^2} p_{\alpha} s^{\alpha}
= \sum_{\alpha_1 = 0}^\infty \sum_{\alpha_2 = 0}^\infty 
p_{\alpha_1,\alpha_2} s_1^{\alpha_1} s_2^{\alpha_2}
\]
where $p_{\alpha} \in \cc^3$ for each $\alpha \in \nn^2$. We write this in the notation from the previous section as $p = (p^{(1)},p^{(2)},p^{(3)}) \in \mathcal{X} = \ell^1 \times \ell^1 \times \ell^1$. 
Observe that 
\[
\lambda_1 s_1 \frac{\partial}{\partial s_1} P(s_1, s_2)
+ \lambda_2 s_2 \frac{\partial}{\partial s_2} P(s_1, s_2) 
= \sum_{\alpha \in \nn^2} (\alpha_1 \lambda_1 + \alpha_2 \lambda_2) p_{\alpha} s^{\alpha},
\]
and that 
\[
f(P(s_1, s_2)) = 
\sum_{\alpha \in \nn^2}
\left(
\begin{array}{c}
\sigma[p^{(2)} - p^{(1)}]_{\alpha} \\
\rho p^{(1)}_{\alpha} - p^{(2)}_{\alpha} - [p^{(1)}*p^{(3)}]_{\alpha} \\
-\beta p^{(3)}_{\alpha} + [p^{(1)}*p^{(2)}]_{\alpha}
\end{array}
\right)
s^{\alpha}.
\]
After matching like powers of $s_1, s_2$, it follows that solutions to Equation~\eqref{eq:lorInvEqn} must satisfy 
\begin{eqnarray*}
(\alpha_1 \lambda_1 + \alpha_2 \lambda_2) p_{\alpha} &=& 
\left(
\begin{array}{c}
	\sigma[p^{(2)} - p^{(1)}]_{\alpha} \\
	\rho p^{(1)}_{\alpha} - p^{(2)}_{\alpha} - [p^{(1)}*p^{(3)}]_{\alpha} \\
	-\beta p^{(3)}_{\alpha} + [p^{(1)}*p^{(2)}]_{\alpha}
\end{array}
\right) \\
&=& \left(
\begin{array}{c}
	\sigma[p^{(2)} - p^{(1)}]_{\alpha} \\
	
\rho p_{\alpha}^{(1)} - p_{\alpha}^{(2)} - p_{0,0}^{(1)} p_{\alpha}^{(3)} - p_{0,0}^{(3)} p_{\alpha}^{(1)} - 
\sum\limits_{\kappa \prec \alpha}
\widehat{\delta}_{\kappa}^{\alpha} p_{\alpha-\kappa}^{(1)} p_{\kappa}^{(3)} \\

-\beta p_{\alpha}^{(3)} + p_{0,0}^{(1)} p_{\alpha}^{(2)} + p_{0,0}^{(2)} p_{\alpha}^{(1)} + 
\sum\limits_{\kappa \prec \alpha}
\widehat{\delta}_{\kappa}^{\alpha} p_{\alpha-\kappa}^{(1)} p_{\kappa}^{(2)} \\
\end{array}
\right)
\end{eqnarray*}
where we define $\widehat{\delta}_{\kappa}^{\alpha}$ by
\[
\widehat{\delta}_{\kappa}^{\alpha} = \begin{cases}
0 & \mbox{ if } \kappa = \alpha \\
0 & \mbox{ if } \kappa = (0,0) \\
1 & \mbox{otherwise}
\end{cases}
\]
Note that the dependence on $p_{\alpha} = (p_{\alpha}^{(1)}, p_{\alpha}^{(2)}, p_{\alpha}^{(3)})$
is linear.  Collecting terms of order $|\alpha| = \alpha_1 + \alpha_2$ on the left and moving lower order
terms on the right gives this dependence explicitly as
{\footnotesize
\begin{eqnarray*}
 \left(
\begin{array}{ccc}
- \sigma -  (\alpha_1 \lambda_1 + \alpha_2 \lambda_2)  &  \sigma &   0 \\
\rho  - p_{0,0}^{(3)}  & - 1 - (\alpha_1 \lambda_1 + \alpha_2 \lambda_2) & - p_{0,0}^{(1)}   \\
 p_{0,0}^{(2)}   &  p_{0,0}^{(1)} &  - \beta -  (\alpha_1 \lambda_1 + \alpha_2 \lambda_2)
\end{array}
\right)
\left(
\begin{array}{c}
p_{\alpha}^{(1)} \\
p_{\alpha}^{(2)} \\
p_{\alpha}^{(3)}
\end{array}
\right) &=& 
 \left(
\begin{array}{ccc}
0 \\ 
\sum\limits_{\kappa \prec \alpha}
\widehat{\delta}_{\kappa}^{\alpha} p_{\alpha-\kappa}^{(1)} p_{\kappa}^{(3)} \\
- \sum\limits_{\kappa \prec \alpha}
\widehat{\delta}_{\kappa}^{\alpha} p_{\alpha-\kappa}^{(1)} p_{\kappa}^{(2)} 
\end{array}
\right)
\end{eqnarray*}
}
which is written more succinctly as
\begin{equation}\label{eq:hom2DFPLor}
\left[
Df(p) - (\alpha_1 \lambda_1 + \alpha_2 \lambda_2) \text{Id}_{\rr^3} \right] p_{\alpha} = q_{\alpha},
\end{equation}
where we define 
\[
q_{\alpha} = 
\left(
\begin{array}{c}
0 \\ 
\sum\limits_{\kappa \prec \alpha}
\widehat{\delta}_{\kappa}^{\alpha} p_{\alpha-\kappa}^{(1)} p_{\kappa}^{(3)} \\
- \sum\limits_{\kappa \prec \alpha}
\widehat{\delta}_{\kappa}^{\alpha} p_{\alpha-\kappa}^{(1)} p_{\kappa}^{(2)} 
\end{array}
\right).
\]
Writing it in this form emphasizes the fact that if $\alpha_1 \lambda_1 + \alpha_2 \lambda_2 \neq \lambda_{1,2}$, then the matrix on the left side of Equation~\eqref{eq:hom2DFPLor} is invertible, and the formal series solution $P$ is defined to all orders. In fact, fixing $N \in \nn$ and solving the homological equations for all $2 \leq |\alpha| \leq N$ leads to our numerical approximation
\[
P^{N}(s_1, s_2) = \sum_{\alpha_1=0}^N \sum_{\alpha_2=0}^{N- \alpha_1} p_{\alpha_1,\alpha_2} s_1^{\alpha_1} s_2^{\alpha_2}.
\]

\begin{remark}(Complex conjugate eigenvalues) When there are complex conjugate eigenvalues 
in fact none of the preceding discussion changes.  The only modification is that,
if we choose complex conjugate eigenvectors, then the coefficients will appear in 
complex conjugate pairs, i.e.\ 
\[
p_{\alpha} = \overline{p_{\alpha}}.
\]
Then taking the complex conjugate variables gives the parameterization of the 
real invariant manifold, 
\[
\hat{P}(s_1, s_2) := P(s_1 + i s_2, s_1 - i s_2),
\]
where $P$ is the formal series defined in the preceding discussion.
For more details see also \cite{MR2728178, MR3207723, MR3068557}.
\end{remark}

%\begin{remark}(Resonance and non-resonance)  When the eigenvalues are 
%resonant, i.e.\ when there is an $\alpha \in \mathbb{N}^2$
%such that 
%\[
%\alpha_1 \lambda_1 + \alpha_2 \lambda_2 = \lambda_{j} \qquad j \in \{1,2\}, 
%\]
%then all is not lost.  In this case we cannot conjugate 
%analytically to the diagonalized linear vector field.  However by modifying the 
%model vector field to include a polynomial term which ``kills'' the resonance 
%the formal computation goes through.  The theoretical details are in \cite{MR1976079},
%and numerical implementation with computer assisted error bounds are discussed and 
%implemented in \cite{parmChristian}.
%\end{remark}

\subsection{Validated error bounds for the Lorenz equations}
The following lemma provides a means to 
obtain mathematically rigorous bounds on the truncation
errors associated with the formal series solutions discussed in the previous 
section.  The result is of an \textit{a-posteriori} variety, i.e.\ we first compute 
an approximation, and then check some conditions associated with the 
approximation.  If the conditions satisfy the hypotheses of the lemma then 
we obtain the desired error bounds. If the conditions are not satisfied, the validation fails and we are unable to make any rigorous statements.
The proof of the lemma is an application of the contraction 
mapping theorem.   

Let $\overbar{a}, \overbar{b}, \overbar{c}$, denote the formal series coefficients, 
computed to $N$-th order using the recursion scheme of the previous 
section, and let 
\[
P^N(s_1, s_2) = \sum_{|\alpha| = 0}^{N} \left(
\begin{array}{c}
\overbar{a}_{\alpha} \\
\overbar{b}_{\alpha} \\
\overbar{c}_{\alpha}
\end{array}
\right) s^{\alpha}.
\]
We treat here only the case where $Df(p)$ is diagonalizable, so that 
\[
Df(p) = Q \Sigma Q^{-1},
\]
with $\Sigma$ the $3\times 3$ diagonal matrix of eigenvalues and 
$Q$ the matrix whose columns are the eigenvectors.  We also assume that 
the eigenvalues are non-resonant, in the sense of Equation \eqref{eq:defRes1}.
We have the following lemma, whose proof is found in \cite{jayAMSnotes}.

\begin{lemma}[A-posteriori analysis for a two dimensional stable/unstable manifold in the Lorenz system]
Let $p \in \mathbb{R}^3$ be a fixed point of the Lorenz system and 
$\lambda_1, \lambda_2 \in \mathbb{C}$ be a pair of non-resonant 
stable (or unstable) eigenvalues of the differential at $p$. 
Assume we have computed $K^N < \infty$ satisfying
\[
K^N > \| Q \| \| Q^{-1}\|  \max_{j = 1,2,3}\sup_{|\alpha| \geq N+1} \left(
\frac{1}{\left| \alpha_1 \lambda_1 + \alpha_2 \lambda_2 - \lambda_j \right|}
\right),
\]
and define the positive constants
\[
Y_0 :=  K^N 
\left( \sum_{|\alpha| = N+1}^{2N} 
\left|[\overbar{a} * \overbar{b}]_{\alpha} \right| + 
\left|[\overbar{a} * \overbar{c}]_{\alpha} \right|
\right),
\]
\[
Z_1 := K^N \left( \sum_{1 \leq |\alpha| \leq N} 2 
\left| \overbar{a}_{\alpha} \right|  
+ \left| \overbar{b}_{\alpha} \right| + \left| \overbar{c}_{\alpha} \right|  \right),
\]
and 
\[
Z_2 := 4 K^N,
\]
and the polynomial 
\[
q(r) := Z_2 r^2 - (1 - Z_1) r + Y_0.
\]
If there exists a $\hat r > 0$ so that $q(\hat r) < 0$, then there exists 
a solution $P$ of Equation  \eqref{eq:invEqnFlows1}, analytic on 
$\mathbb{D}^2$, with
\[
\sup_{|s_1|, |s_2| < 1}  
\left\| P(s_1, s_2) - P^N(s_1, s_2) \right\|_{\mathbb{C}^3} \leq \hat{r}.
\]
\end{lemma}

\bigskip

\noindent \textbf{Example:}
Consider the two-dimensional stable manifold at the origin at the 
classical parameter values $\sigma = 10$, $\beta = 8/3$ and $\rho = 28$ in the Lorenz system. Using validated numerical algorithms discussed in  \cite{MR2652784}, 
and implemented in IntLab \cite{Ru99a}, we compute the following 
first order data.  The unstable eigenvalue satisfies
\[
\lambda^u \in [  11.82772345116345,  11.82772345116347],
\]
while the stable eigenvalues satisfy
\[
\lambda_1^s \in [  -2.66666666666667,  -2.66666666666666] ,
\]
and
\[
\lambda_2^s \in [ -22.82772345116347, -22.82772345116345].
\]
We compute corresponding eigenvectors that satisfy the inclusions
\[
\xi^u \in \left(
\begin{array}{c}
\,  [  -0.41650417819291,  -0.41650417819290]   \\
\, [  -0.90913380178490,  -0.90913380178489] \\
\, [  -0.00000000000001,   0.00000000000001]  \\
\end{array}
\right),
\]
\[
\xi_1^s = \left(
\begin{array}{c}
\,0 \\
\,0  \\
\,1
\end{array}
\right),
\]
and
\[
\xi_2^s \in \left(
\begin{array}{c}
\, [  -0.61481678521648,  -0.61481678521647]  \\
\, [   0.78866996938902,   0.78866996938903]  \\
\,0   \\
\end{array}
\right).
\]
Recursively solving the homological equations to order $N = 50$
yields the approximating polynomial, and also rules out resonances up to order fifty.  

Now, suppose that $\alpha \in \mathbb{N}^2$ with $|\alpha| \geq 51$.
Since $N = 50 > |\lambda_2^s| > |\lambda_1^s|$ we have that 
\[
\frac{1}{|\alpha_1 \lambda_1^s + \alpha_2 \lambda_2^s - \lambda_{1}^s|} \leq 
\frac{1}{|(\alpha_1 + \alpha_2) \lambda_1^s - \lambda_{1}^s|} = \frac{1}{|\alpha_1 + \alpha_2 - 1||\lambda_1^s|} \leq \frac{1}{50 |\lambda_1^s|} \leq 0.0075,
\]
\[
\frac{1}{|\alpha_1 \lambda_1^s + \alpha_2 \lambda_2^s - \lambda_{2}^s|} \leq 
\frac{1}{|(\alpha_1 + \alpha_2) \lambda_1^s - \lambda_{2}^s|} = \frac{1}{(\alpha_1 + \alpha_2) |\lambda_1^s| - |\lambda_2^s|} 
\leq \frac{1}{50 |\lambda_1^s|- |\lambda_2^s|} \leq 0.0089,
\]
and 
\[
\frac{1}{|\alpha_1 \lambda_1^s + \alpha_2 \lambda_2^s - \lambda^u|} \leq 
\frac{1}{|(\alpha_1 + \alpha_2) |\lambda_1^s| + |\lambda^u|} 
\leq \frac{1}{50 |\lambda_1^s|+ |\lambda^u|} \leq 0.0068.
\]
Thus, there are no resonances at any order, and from the enclosures of the eigenvectors we may take 
\[
K^N = 0.009.
\]
We scale the slow eigenvector to have length $15$ and the 
fast eigenvector to have length $1.5$ (as the difference in the 
magnitudes of the eigenvalues is about ten).
We obtain a validated contraction mapping error bound of 
$7.5 \times 10^{-20}$, which is below machine precision, but 
we need order $N = 50$ with this choice of scalings in 
order to get 
\[
Z_1 = 0.71 < 1.
\]
We note that we could take lower order and smaller scalings to validate a smaller
portion of the manifold.
The two dimensional validated local stable manifold at the origin is the one illustrated 
in Figures~\ref{fig:slowstable} and \ref{fig:faststable}.

\section{Validated integration of analytic surfaces} \label{sec:rigorousIntegration}
\label{sec:integrator}
%\subsection{Validated Taylor-Taylor Integration for Analytic Vector Fields}
Let $\Omega \subset \mathbb{R}^n$ be an open set and 
$f: \Omega \to \rr^n$ a real analytic vector field. Consider 
$\gamma \colon [-1,1]^{d-1} \to \rr^n$ a parameterized 
manifold with boundary. 
Recalling the summary of our scheme from Section~\ref{sec:intro}, we 
have in mind that $\gamma$ is a chart parameterizing a portion of
the boundary of $W^u_{\text{\tiny loc}}(p_0)$, transverse to the flow. 
Assume moreover that $\gamma \in \mathcal{B}_{d-1}^1$,
so that the Taylor coefficients of $\gamma$ are absolutely summable.
Define $\Gamma: [-1, 1]^d \to \rr^n$ the advected image of 
$\gamma$ given by $\Phi(\gamma(s),t) = \Gamma(s,t)$. 
We are especially interested in the case where 
$\Gamma \in \mathcal{B}_d^1$, however this will be a conclusion of 
our computer assisted argument rather than an assumption.

\subsection{Validated single step Taylor integrator}
\label{sec:onestep}
Numerical Taylor integration of the manifold $\gamma$ requires 
a finite representation, which we now describe.
Assume that $\gamma$ is specified as a pair, $(\hat{a},r_0)$, where $\hat{a}$ is a 
finite $\ell^1_{d-1}$ approximation of $\TT{\gamma}$ (i.e. a polynomial), 
and $r_0 \geq 0$ is a scalar error bound (norm in the $\ell^1_{d-1}$ topology). 
Second, we note that there is a technical issue of dimensions. To be more precise, 
let $a = \{a_{m,\alpha}\} = \mathcal{T}(\Gamma)$ denote the $d$-variable Taylor 
coefficients for the evolved surface. Recall from Section~\ref{sec:background} that the
 double indexing on $a$ allows us to distinguish between coefficients in the space or 
 time ``directions''. It follows that the appropriate space in which to seek solutions is 
 the product space $(\ell^1_{d})^n$. Strictly speaking however, $\TT{\gamma}$ is a 
 coefficient sequence in $(\ell^1_{d-1})^n$. Nevertheless, the 
fact that $\Gamma(s,0) = \gamma(s)$ implies that 
$\TT{\gamma} = \{a_{0,\alpha}\}_{\alpha \in \nn^{d-1}}$, and
 this suggests working in $\mathcal{X} = (\ell^1_{d})^n$ with the understanding that $\TT{\gamma}$ 
 has a natural embedding in $\mathcal{X}$ by padding with zeros in the time direction. 

In this context, our one-step integration scheme is an algorithm which takes input 
$(\hat{a},r_0,t_0)$ and produces output $(\overbar{a},r,\tau)$ satisfying
\begin{itemize}
	\item $\norm{\hat{a} - \TT{\gamma}}_{\mathcal{X}} < r_0$
	\item $\norm{\overbar{a} - \TT{\Gamma}}_{\mathcal{X}} < r$
\end{itemize}
In particular, we obtain a polynomial approximation, $\overbar{\Gamma} = \itt{\overbar{a}}$, 
which satisfies $\norm{\overbar{\Gamma}(s,t)-\Gamma(s,t)}_{\infty} < r$ for 
every $(s,t) \in \mathbb{D}^{d-1} \times [t_0, t_0 + \tau]$. For ease of exposition, 
we have also assumed that $f$ is autonomous therefore we may take $t_0 = 0$ without 
loss of generality. 

\subsubsection*{Numerical approximation}
The first step is a formal series calculation, which we validate a-posteriori.
Suppose $\tau > 0$ and $\Gamma$ 
satisfy the initial value problem
\begin{equation}
\label{eq:ivp}
\frac{d \Gamma}{dt} = f(\Gamma(s,t)) \qquad \Gamma(s,0) = \gamma(s)
\end{equation}
for all $(s,t) \in \mathbb{D}^{d-1} \times [0,\tau)$.  Write 
\[
\Gamma(s,t) = \sum_{m \in \nn} \sum_{\alpha \in \nn^{d-1}}
a_{m,\alpha} s^{\alpha}t^{m}.
\]
Evaluating both sides of \eqref{eq:ivp} leads to 
\begin{align}
\frac{\partial \Gamma}{\partial t} = & \sum_{m \in \nn} \sum_{\alpha \in \nn^{d-1}} ma_{m,\alpha} s^{\alpha}t^{m-1} \\
f(\Gamma(s,t)) = & \sum_{m \in \nn} \sum_{\alpha \in \nn^{d-1}} c_{m,\alpha} s^{\alpha}t^m,
\end{align}
where each $c_{m-1,\alpha}$ depends only on lower order terms in the set $\{ a_{j,\kappa} \colon (j,\kappa) \prec (m-1,\alpha)\}$. Satisfaction of the initial condition in \eqref{eq:ivp} implies $\Gamma(s,0) = \gamma(s)$ which leads to the relation on the coefficient level given by
\begin{equation}
\{a_{0,\alpha}\}_{\alpha \in \nn^{d-1}} =  \hat{a}. 
\end{equation}
Moreover, uniqueness of solutions to \eqref{eq:ivp} allows us to conclude that $\mathcal{T}(f \circ \Gamma) = \mathcal{T}(\frac{\partial \Gamma}{\partial t})$. This gives a recursive characterization for $a$ given by 
\begin{equation}
\label{eq:recursion}
ma_{m,\alpha} = c_{m-1,\alpha} \qquad m \geq 1,
\end{equation}
which can be computed to arbitrary order. Our approximation is now obtained by fixing a degree, $(m,\alpha) \in \nn \times \nn^{d-1}$, and computing $\overbar{a}_{j,\kappa}$ recursively for all $(j,\kappa) \prec (m,\alpha)$. This yields a numerical approximation to $(m,\alpha)^{th}$ degree Taylor polynomial for $\Gamma$ whose coefficients are given by $a^{m \alpha}$, and we define $\overbar{\Gamma} = \itt{\overbar{a}}$ to be our polynomial approximation of $\Gamma$. 

\remark{It should be emphasized that there is no requirement to produce the finite approximation using this recursion. In the case where it makes sense to use a Taylor basis for $C^\omega(\mathbb{D})$, this choice minimizes the error from truncation. However, the validation procedure described below does not depend on the manner in which the numerics were computed. Moreover, for a different choice of basis (e.g. Fourier, Chebyshev) there is no recursive structure available and an approximation is assumed to be provided by some means independent of the validation.}

\subsubsection*{Rescaling time}
Next, we rescale $\Gamma$ to have as its domain the unit polydisk. 
This rescaling provides control over the decay rate of the 
Taylor coefficients of $\Gamma$, giving a kind of numerical stability. As already mentioned
above, $\tau$ is an approximation/guess for the radius of convergence of $\Gamma$
before rescaling. In general $\tau$ is a-priori unknown and difficult to estimate for even a 
single initial condition, much less a higher dimensional surface of initial conditions. 
Moreover, suppose $\tau$ could be computed exactly by some method.  
Then $\Gamma$ would be analytic on the polydisc 
$\mathbb{D}^{d-1} \times \mathbb{D}_{\tau}$ 
which necessitates working in a weighted $\ell^1$ space. The introduction of weights 
to the norm destabilizes the numerics.  

Let $\gamma \in \mathcal{B}_d^1$ denote manifold of initial conditions of the local 
unstable manifold.   
Simply stated, the idea is to compute first the Taylor coefficients with no rescaling
and examine the numerical growth rate of the result.  The coefficients will decay/grow
exponentially with some rate we approximate numerically.  Growth suggests we are
trying to take  too long a time step -- decay suggests too short.  In either case we rescale
so that the resulting new growth rate makes our last coefficients small relative to the precision 
of the digital computer. 
 
More precisely, let $\mu$ denote the machine unit for a fixed precision floating point 
implementation(e.g. $\mu \approx 2^{-54} \approx 2.44 \times 10^{-16}$ for double 
precision on contemporary $64$ bit micro-processor architecture) 
and consider our initial finite numerical 
approximation as a coefficient vector of the form $\overbar{a} \approx a = \TT{\Gamma}$. 
Suppose it has degree $(M,N) \in \nn \times \nn^{d-1}$, and rewrite this polynomial 
after ``collapsing'' onto the time variable as follows: 
\[
\overbar{\Gamma}(s,t) =  \sum_{m = 0}^{M} \sum_{\alpha \prec N} a_{m,\alpha} s^{\alpha}t^m = 
 \sum_{m = 0}^M p_m(s)t^m 
\]
where $p_m(s)$ is a polynomial approximation for the projection of $\Gamma$ onto the $m^{th}$ 
term in the time direction. Note that $p_m$ may be identified by its coefficient vector 
given by $\TT{p_m} = \{\overbar{a}_{m,\kappa}\}_{\alpha \prec N}$. Now, we define 
\[
w = \max 
\left\{ 
\sum_{\alpha \prec N} \left| \overbar{a}^{(1)}_{M,\alpha} \right|,\dots, \sum_{\alpha \prec N} \left| \overbar{a}^{(n)}_{M,\alpha} \right|
\right\}
= \norm{\TT{p_M}}_{\mathcal{X}}
\]
and set  
\[
\label{eq:timerescaling}
L = (\frac{\mu}{w})^{\nicefrac{1}{M}}
\]
an approximation of $\tau$. In other words, we choose a time rescaling, $L$, which tunes 
our approximation so that for each coordinate of $\overbar{a}$, the $M^{th}$ coefficient 
(in time) has norm no larger than machine precision. This is equivalent to flowing by the 
time-rescaled vector field $f_L(x) = Lf(x)$. The standard (but crucial) observation 
is that the trajectories of the time rescaled vector field 
are not changed.  Therefore the advected image of $\gamma$ by $f_L$ still lies in the 
unstable manifold. It is also this time rescaling which permits us to seek solutions for $t \in [-1,1]$ since the time-1 map for the rescaled flow is equivalent to the time-$L$ map for the unscaled map. 

\subsubsection*{Error bounds for one step of integration}
Now define a function $F \in C^1(\mathcal{X})$ by
\begin{equation}
\label{eq:defF}
[F(x)]_{m,\alpha} = 
\left\{
\begin{array}{cc}
x_{0,\alpha} -  [\mathcal{T}(\gamma)]_{\alpha}  & m = 0\\
mx_{m,\alpha} - c_{m-1,\alpha} & m \geq 1
\end{array}
\right.
\end{equation}
where the coefficients $c_{m-1,\alpha}$ are given by $\mathcal{T} \circ f \circ \mathcal{T}^{-1}(x)$. 
Intuitively, $F$ measures how close the analytic function defined by $x$ comes to satisfying 
\eqref{eq:ivp}. Specifically, we notice that $F(x) = 0$ if and only if $\mathcal{T}^{-1}(x) = \Gamma$ 
or equivalently, $F(x) = 0$ if and only if $x = a$. 
We prove the existence of a unique solution of this equation in the infinite sequence space 
$\mathcal{X} = \left(\ell_d\right)^n$.  The corresponding function $\Gamma \in \mathcal{B}_d^1$
solves the initial value problem for the initial data specified by $\gamma$.  Moreover, the final 
manifold given by 
\[
\hat \gamma(s) := \Gamma(s, 1),
\]
has $\hat \gamma \in \mathcal{B}_{d-1}^1$.  Then the final condition $\hat \gamma$
is a viable initial condition for the next stage of validated integration. Further details are included in Section~\ref{sec:time stepping}.  

Finally, given an approximate solution of the zero finding problem for Equation~\eqref{eq:defF}
we develop a-posteriori estimates which allow us to conclude that there is a true solution 
nearby using Proposition~\ref{prop:newton}.  This involves choosing an approximate derivate
$A^\dagger$, an approximate inverse $A$, and derivation of the $Y_0$, $Z_0$, $Z_1$, 
and $Z_2$ error bounds for the application at hand.
The validation method is best illustrated in a particular example which will be taken up in the next section. 

\subsection{Examples and performance: one step of integration}
Recall the Lorenz field defined in Equation \eqref{eq:LorenzField}.
For the classical parameter values of $\rho = 28$, $\sigma = 10$ and $\beta = 8/3$
the three equilibria are hyperbolic with either two-dimensional stable or a two-dimensional 
unstable manifold. Then, in the notation of the previous section we have that 
$d = 2$ and $\mathcal{X} = \ell^1_2 \times \ell^1_2 \times \ell^1_2$. The boundaries of these 
manifolds are one-dimensional arcs whose advected image under the flow is a 
two-dimensional surface. We denote each as a power series by   
\begin{align}
\gamma(s) = & \sum_{\alpha =0}^{\infty} 
\left(
\begin{array}{c}
a_{0,\alpha}\\
b_{0,\alpha}\\
c_{0,\alpha}
\end{array}
\right)
s^\alpha \\
\Gamma(s,t) = & \sum_{m =0}^{\infty}  \sum_{\alpha =0}^{\infty} 
\left(
\begin{array}{c}
a_{m,\alpha}\\
b_{m,\alpha}\\
c_{m,\alpha}
\end{array}
\right)
s^\alpha t^m 
\end{align}
where $(s,t) \in [-1,1]^2$.
We write $\TT{\Gamma} = (a,b,c)$ and obtain its unique characterization in $\mathcal{X}$ by applying the recursion in \eqref{eq:recursion} directly which yields the relation on the coefficients given by
\begin{equation}
\label{eqn:1}
\left(
\begin{array}{cc}
a_{m+1,\alpha} \\
b_{m+1,\alpha} \\
c_{m+1,\alpha}
\end{array}
\right)
 = \frac{L}{m+1} 
\left(
\begin{array}{cc}
\sigma(b_{m,\alpha} - c_{m,\alpha})\\
{[\rho a - a*c]}_{m,\alpha} - b_{m,\alpha} \\
{[a*b]}_{m,\alpha} - \beta c_{m,\alpha}
\end{array}
\right).
\end{equation}
where $L$ is the constant computed in Equation~\eqref{eq:timerescaling}.
This recursion is used to compute a finite approximation denoted by 
$(\overbar{a},\overbar{b},\overbar{c}) \in \mathcal{X}$ with order $(M,N) \in \nn^2$. Next, we define the map 
$F\in C^1(\mathcal{X})$ as described in \eqref{eq:defF} and denote it by 
$F(x,y,z) = \left( F_1(x,y,z),F_2(x,y,z),F_3(x,y,z) \right)^T$ where $(x,y,z) \in \mathcal{X}$. 

Now, express $DF(\overbar{a},\overbar{b},\overbar{c})$ as a $3 \times 3$ block matrix of operators on $\ell^1$. 
Each block is an element in $\mathcal{L}(\ell^1)$, and its action on an arbitrary vector 
$h\in \ell^1$ is described in terms of the operators from Section~\ref{sec:linearAlgebra} as follows
\begin{align*}
D_1F_1(\overbar{a},\overbar{b},\overbar{c}) \cdot h & = h' + \sigma L \eta(h) \\
D_2F_1(\overbar{a},\overbar{b},\overbar{c}) \cdot h & = -\sigma L \eta(h) \\
D_3F_1(\overbar{a},\overbar{b},\overbar{c}) \cdot h & = 0 \\
D_1F_2(\overbar{a},\overbar{b},\overbar{c}) \cdot h & = -L\eta(\rho h - c*h) \\
D_2F_2(\overbar{a},\overbar{b},\overbar{c}) \cdot h & = h' + L\eta(h) \\
D_3F_2(\overbar{a},\overbar{b},\overbar{c}) \cdot h & = L\eta(a*h) \\
D_1F_3(\overbar{a},\overbar{b},\overbar{c}) \cdot h & = -L\eta(b*h) \\
D_2F_3(\overbar{a},\overbar{b},\overbar{c}) \cdot h & = -L\eta(a*h) \\
D_3F_3(\overbar{a},\overbar{b},\overbar{c}) \cdot h & = h' + \beta L \eta(h). \\
\end{align*}
Recalling the notation from Section~\ref{sec:linearAlgebra}, we will denote these nine operators by 
\[
DF_{(ij)}(\overbar{a},\overbar{b},\overbar{c}) = D_jF_i(\overbar{a},\overbar{b},\overbar{c}). 
\]

\subsection{A-posteriori analysis for the rigorous integrator in Lorenz}
We now describe the application of the a-posteriori validation method described in Section \ref{sec:aPos} to the rigorous integrator for the Lorenz example. This requires specifying appropriate linear operators, $A,A^{\dagger}$, and constants, $r,Y_0,Z_0,Z_1,Z_2$, which allow application of Proposition \ref{prop:newton} for the Lorenz integrator. The error bounds in the examples of Section \ref{sec:Lorenzresults} are then obtained by applying the Radii polynomial method described in Remark~\ref{rmk:radPoly}. 

\subsubsection*{Defining $A^{\dagger}$}
We specify $A^{\dagger}$ to be an approximation of $DF(\overbar{a},\overbar{b},\overbar{c})$ which is diagonal in the ``tail''. Specifically, $DF_{(ij)}^{MN}(\overbar{a},\overbar{b},\overbar{c})$ denotes the truncation of $DF(\overbar{a},\overbar{b},\overbar{c})$ and we define $A^{\dagger}$ to be the $3 \times 3$ block of operators whose action on a 
vector $h \in \ell^1$ is given by 
\[
[A_{(ij)}^{\dagger}\cdot h]_{m,\alpha} = 
\left\{
\begin{array}{cc}
[DF_{(ij)}^{MN}(\overbar{a},\overbar{b},\overbar{c}) \cdot h]_{m,\alpha} & (m,\alpha) \prec(M,N) \\
m h_{m,\alpha} & (m,\alpha) \succ(M,N), \ i = j \\
0 & \text{otherwise}
\end{array}
\right.
\]
In other words, the finite part of the action of $A^{\dagger}$ is determined by the finite part of $DF(\overbar{a},\overbar{b},\overbar{c})$ and the infinite part along the diagonal is given by the derivative operator defined in Section \ref{sec:linearAlgebra}. 

\subsubsection*{Defining $A$}
The operator $A$ is an approximation for the inverse of $DF(\overbar{a},\overbar{b},\overbar{c})$. For this example, we have used an approximate inverse for $A^{\dagger}$ instead which motivates our choice for the tail of $A^{\dagger}$. Specifically, the finite part of $A$ is obtained by numerically inverting $DF^{MN}(\overbar{a},\overbar{b},\overbar{c})$, and $A$ acts on the tail of vectors in $\mathcal{X}$ by scaling the diagonal coordinates by $\frac{1}{m}$. 

\subsubsection*{$Y_0$ bound}
We decompose $F$ as 
\[
F (\overbar{a},\overbar{b},\overbar{c}) = F^{MN} (\overbar{a},\overbar{b},\overbar{c})+ F^{\infty} (\overbar{a},\overbar{b},\overbar{c}),
\]
where $F^{MN}$ and $F^\infty$ are as defined in Equation~\eqref{eq:truncDef}.
Note that if $(m,\alpha) \succ (M,N)$, then $\overbar{a}_{m,\alpha} = \overbar{b}_{m,\alpha} = \overbar{c}_{m,\alpha} = 0$, and thus the only nonzero contributions to $F(\overbar{a},\overbar{b},\overbar{c})^{\infty}$ are due to higher order terms from $\TT{\gamma}$, or the Cauchy products of low order terms due to the nonlinearity. Specifically, we have the following: 
\begin{align*}
{[F_1^{\infty}(\overbar{a},\overbar{b},\overbar{c})]_{m,\alpha}} & = 
\left(
\begin{array}{cc}
-[a]_{0,\alpha} & m = 0 \\
0 & \text{otherwise} 
\end{array}
\right) \\
{[F_2^{\infty}(\overbar{a},\overbar{b},\overbar{c})]_{m,\alpha}} & = 
\left(
\begin{array}{cc}
{[\overbar{a}*\overbar{c}]_{0,\alpha}}-[b]_{0,\alpha} & m = 0 \\
{[\overbar{a}*\overbar{c}]_{m,\alpha}} & \text{otherwise} 
\end{array}
\right) \\
{[F_3^{\infty}(\overbar{a},\overbar{b},\overbar{c})]_{m,\alpha}} &  = 
\left(
\begin{array}{cc}
{[\overbar{a}*\overbar{b}]_{0,\alpha}}-[c]_{0,\alpha} & m = 0 \\
{[\overbar{a}*\overbar{b}]_{m,\alpha}} & \text{otherwise} 
\end{array}
\right) 
\end{align*}
where we also note that for all $(m,\alpha) \succ (2M,2N)$, we have $[\overbar{a}*\overbar{c}]_{m,\alpha} = 0 = [\overbar{a}*\overbar{b}]_{m,\alpha}$. Recalling the definition of the operator $A$, we also have $A_{(ij)}^{m \alpha} = 0$ for $i \neq j$ and $(m,\alpha) \succ (M,N)$. Combining these observations leads to defining the following constants 
\[
Y_1 = \norm{[A^{MN}F^{MN}(\overbar{a},\overbar{b},\overbar{c})]_1}_{1} + \norm{a_{0,\alpha}^{\infty}}_{1}
\]
\[
Y_2 = \norm{[A^{MN}F^{MN}(\overbar{a},\overbar{b},\overbar{c})]_2}_{1} + \sum_{m = M+1}^{2M} \frac{1}{m} \sum_{\alpha = N+1}^{2N} [\overbar{a}*\overbar{c}]_{m,\alpha} + \norm{b_{0,\alpha}^{\infty}}_{1}
\]
\[
Y_3 = \norm{[A^{MN}F^{MN}(\overbar{a},\overbar{b},\overbar{c})]_3}_{1} + \sum_{m = M+1}^{2M} \frac{1}{m} \sum_{\alpha = N+1}^{2N} [\overbar{a}*\overbar{b}]_{m,\alpha} + \norm{c_{0,\alpha}^{\infty}}_{1}
\]
and we conclude that 
\begin{equation}
\label{sec:Lorenz_Y0}
\norm{AF(\overbar{a},\overbar{b},\overbar{c})}_{\mathcal{X}} \leq \max \{Y_1,Y_2,Y_3\} := Y_0.
\end{equation}

\subsubsection*{$Z_0$ bound}
We will define the constant $Z_0 := \norm{\text{Id}_{\mathcal{X}}^{MN} - A^{MN}DF^{MN}(\overbar{a},\overbar{b},\overbar{c})}_{\mathcal{X}}$ and we claim that $\norm{\text{Id}_{\mathcal{X}} - AA^{\dagger}}_{\mathcal{X}} \leq Z_0$. This follows directly from the computation
\[
AA^{\dagger} = 
\left(
\begin{array}{cc}
A^{MN}DF^{MN}(\overbar{a},\overbar{b},\overbar{c}) & 0 \\
0 & \text{Id}_{\mathcal{X}} 
\end{array}
\right)
\]
where the expression on the right is a block matrix of operators in $\mathcal{L}(\mathcal{X})$.
Therefore, we have 
\[
\norm{\text{Id}_{\mathcal{X}} - AA^{\dagger}}_{\mathcal{X}} = \norm{\text{Id}_{\mathcal{X}}^{MN} - A^{MN}DF^{MN}(\overbar{a},\overbar{b},\overbar{c})}_{\mathcal{X}} = Z_0
\]
and we note that our choice of $A$ and $A^\dagger$ are (partially) motivated by requiring that this estimate reduces to a finite dimensional matrix norm which is rigorously computable using interval arithmetic. 

\subsubsection*{$Z_1$ bound}
We define the $Z_1$ constant for Lorenz  
\[
Z_1 := \frac{L}{M} \max\{
2 \sigma, \rho + \norm{\overbar{c}}_1 + 1 + \norm{\overbar{a}}_{1}, \norm{\overbar{b}}_1 + \norm{\overbar{a}}_1 + \beta 
\}
\]
and recalling Proposition \ref{prop:newton} we must prove that $\norm{A(A^{\dagger} - DF(\overbar{a},\overbar{b},\overbar{c}))}_{\mathcal{X}} \leq Z_1$. \\
Suppose $(u,v,w)^T$ is a unit vector in $\mathcal{X}$ and define 
	\begin{align*}
	(u_1,v_1,w_1) & = (A^{\dagger} - DF(\overbar{a},\overbar{b},\overbar{c}))\cdot (u,v,w)^T \\
	& = 
	\left(
	\begin{array}{ccc}
	A_{(11)}^{\dagger} - D_1F_1(\overbar{a},\overbar{b},\overbar{c}) & A_{(12)}^{\dagger} - D_2F_1(\overbar{a},\overbar{b},\overbar{c}) & A_{(13)}^{\dagger} - D_3F_1(\overbar{a},\overbar{b},\overbar{c}) \\
	A_{(21)}^{\dagger} - D_1F_2(\overbar{a},\overbar{b},\overbar{c}) & A_{(22)}^{\dagger} -  D_2F_2(\overbar{a},\overbar{b},\overbar{c}) & A_{(23)}^{\dagger} - D_3F_2(\overbar{a},\overbar{b},\overbar{c}) \\
	A_{(31)}^{\dagger} - D_1F_3(\overbar{a},\overbar{b},\overbar{c}) & A_{(32)}^{\dagger} -  D_2F_3(\overbar{a},\overbar{b},\overbar{c}) & A_{(33)}^{\dagger} - D_3F_3(\overbar{a},\overbar{b},\overbar{c}) \\
	\end{array}
	\right)
	\cdot 
	\left(
	\begin{array}{c}
	u \\
	v \\
	w
	\end{array}
	\right)
	\end{align*}
and note that if $(m,\alpha) \prec (M,N)$ then 
$(A_{(ij)}^{\dagger})^{m\alpha} = (D_jF_i^{MN}(\overbar{a},\overbar{b},\overbar{c}))^{m\alpha}$ for all 
$i,j \in \{1,2,3\}$ and thus $(u_1,v_1,w_1)^{MN} = (0,0,0)$. 

\noindent {\bf Computing ${\bf u_1}$}: 
	Recalling the expressions for the blocks of $DF(\overbar{a},\overbar{b},\overbar{c})$ we have
	\begin{eqnarray*}
		D_1F_1(\overbar{a},\overbar{b},\overbar{c})\cdot u = & u' + \sigma L \eta(u) \\
		D_2F_1(\overbar{a},\overbar{b},\overbar{c}) \cdot v = & -\sigma L \eta(v) \\
		D_3F_1(\overbar{a},\overbar{b},\overbar{c}) \cdot w = & 0
	\end{eqnarray*}
	After canceling the contribution from $A_{(11)}^{\dagger}$ and summing the remainders we obtain the expression for $u_1$
	\[
u_1 = L \sigma \eta (u-v)^{\infty}.
	\]
	{\bf Computing $\bf {v_1}$}: \\ 
	We proceed similarly with the second row in order to compute $v_1$. 
	\begin{eqnarray*}
		D_1F_2(\overbar{a},\overbar{b},\overbar{c})\cdot u = & -L \eta(\rho u - \overbar{c}*u) \\
		D_2F_2(\overbar{a},\overbar{b},\overbar{c}) \cdot v = & L \eta(v) \\
		D_3F_2(\overbar{a},\overbar{b},\overbar{c}) \cdot w = & L \beta \eta(w)
	\end{eqnarray*}
	and canceling the diagonal and adding as before we obtain
	\[
	v_1 = L\eta(-\rho u - \overbar{c}*u + v + \overbar{a}*w)^{\infty}.
	\]
	{\bf Computing $\bf {w_1}$}: \\ 
	Computing along the third row in the same manner we have
	\begin{eqnarray*}
		D_1F_3(\overbar{a},\overbar{b},\overbar{c})\cdot u = & -L\eta(\overbar{b}*u)  \\
		D_2F_3(\overbar{a},\overbar{b},\overbar{c}) \cdot v = & -L\eta(\overbar{a}*v) \\
		D_3F_3(\overbar{a},\overbar{b},\overbar{c}) \cdot w = & L\beta \eta(w)
	\end{eqnarray*}
	and thus after cancellation
	\[
	w_1 = L \eta (\overbar{b}*u - \overbar{a}*v + \beta w)^{\infty}.
	\]
	Next, we define $(u_2,v_2,w_2) \in \mathcal{X}$ by 
	\[
	(u_2,v_2,w_2)^T = A\cdot (u_1,v_1,w_1)^T = 
	\left(
	\begin{array}{ccc}
	A_{(11)} & A_{(12)} & A_{(13)} \\
	A_{(21)} & A_{(22)}  & A_{(23)} \\
	A_{(31)} & A_{(32)}  & A_{(33)} \\
	\end{array}
	\right)
	\cdot 
	\left(
	\begin{array}{c}
	u_1\\
	v_1 \\
	w_1
	\end{array}
	\right)
	\]
	and recall that $(u_1,v_1,w_1)^{MN} = (0,0,0)$ so if $(m,\alpha) \prec (M,N)$, then any non-zero contributions to $[(u_2,v_2,w_2)^{MN}]_{m,\alpha}$ must come from $A^{j\kappa}$ where $(j,\kappa) \succ (M,N)$. However, since each block of $A$ is diagonal in the tail, it follows that there are no non-zero contributions from these terms so we conclude that $(u_2,v_2,w_2)^{MN} = (0,0,0)$ as well. Moreover, if $i \neq j$ and $(m,\alpha) \succ (M,N)$, then $A_{(ij)}^{m\alpha} = 0$ which yields bounds on  $\norm{u_2}_1,\norm{v_2}_1,\norm{w_2}_1$ given by
	\begin{align*}
	\norm{u_2}_1
	& = \norm{A_{(11)} \cdot u_1 + A_{(12)} \cdot v_1 + A_{(13)} \cdot w_1}_1\\
	& \leq \norm{\underbrace{A_{(11)} \cdot u_1}_{L \sigma \eta (u-v)^{\infty}}}_1 
	+ \norm{\underbrace{A_{(12)} \cdot v_1}_{0_{\ell^1}}}_1
	+ \norm{\underbrace{A_{(13)} \cdot w_1}_{0_{\ell^1}}}_1 \\
	& \leq \frac{L \sigma}{M} \norm{u-v}_1\\
	& \leq \frac{2 L \sigma}{M}
	\end{align*}
	
	\begin{align*}
	\norm{v_2}_1 
	& = \norm{A_{(11)} \cdot u_1 + A_{(12)} \cdot v_1 + A_{(13)} \cdot w_1}_1\\
	& \leq \norm{\underbrace{A_{(11)} \cdot u_1}_{0_{\ell^1}}}_1 
	+ \norm{\underbrace{A_{(12)} \cdot v_1}_{ L\eta(-\rho u - \overbar{c}*u + v + \overbar{a}*w)^{\infty}}}_1
	+ \norm{\underbrace{A_{(13)} \cdot w_1}_{0_{\ell^1}}}_1 \\
	& \leq \frac{L}{M} \norm{-\rho u - \overbar{c}*u + v + \overbar{a}*w}_1 \\
	& \leq \frac{L}{M}(\rho + \norm{\overbar{c}}_1 + 1 + \norm{\overbar{a}}_1)
	\end{align*}
	
	\begin{align*}
	\norm{w_2}_1
	& = \norm{A_{(11)} \cdot u_1 + A_{(12)} \cdot v_1 + A_{(13)} \cdot w_1}_1\\
	& \leq \norm{\underbrace{A_{(11)} \cdot u_1}_{0_{\ell^1}}}_1 
	+ \norm{\underbrace{A_{(12)} \cdot v_1}_{0_{\ell^1}}}_1
	+ \norm{\underbrace{A_{(13)} \cdot w_1}_{L \eta (\overbar{b}*u - \overbar{a}*v + \beta w)^{\infty}}}_1 \\
	& \leq \frac{L}{M} \norm{\overbar{b}*u - \overbar{a}*v + \beta w}_1 \\
	& \leq \frac{L}{M}(\norm{\overbar{b}}_1 + \norm{\overbar{a}}_1 + \beta)
	\end{align*}
	where we have used the estimates given in Proposition \ref{prop:operatorbounds}.
	Since $(u,v,w) \in \mathcal{X}$ was an arbitrary unit vector we conclude from the definition of the operator norm on $\mathcal{X}$ that 
	\[
	\norm{A(A^{\dagger} - DF(\overbar{a},\overbar{b},\overbar{c}))}_{\mathcal{X}} \leq 
	\frac{L}{M} \max \left\{
	2 \sigma, \rho + \norm{\overbar{c}}_1+ 1 + \norm{\overbar{a}}_1, \norm{\overbar{b}}_1 + \norm{\overbar{a}}_1 + \beta 
	\right\} = Z_1.
	\]

\subsubsection*{$Z_2$ bound}
Finally,  define  
\[
Z_2 := 2L \max \left\{\norm{A^{MN}}_{\mathcal{X}}, \frac{1}{M} \right\},
\]  
and consider $(x,y,z) \in B_{r}(\overbar{a},\overbar{b},\overbar{c})$.  Take $(u,v,w)^T \in \mathcal{X}$ 
a unit vector as above. Using the definition of $DF$ we express $\norm{(DF(x,y,z) - (DF(\overbar{a},\overbar{b},\overbar{c}))	
	\cdot (u,v,w)^T}_{\mathcal{X}}$ explicitly as 
	\[
	\left|\left|
	\left(
	\begin{array}{cc}
	0 \\
	L \eta ((z-\overbar{c})*h) + L \eta((x-\overbar{a})*w) \\
	-L \eta((y-\overbar{b})*h) - L\eta((x-\overbar{a})*v)
	\end{array}
	\right)
	\right|\right|
	_{\mathcal{X}}
	\leq 
	L\left|\left|
	\left(
	\begin{array}{cc}
	0 \\
	\eta (z-\overbar{c}) + \eta(x-\overbar{a}) \\
	-\eta(y-\overbar{b}) - \eta(x-\overbar{a})
	\end{array}
	\right)
	\right|\right|
	_{\mathcal{X}}
	\leq 2Lr
	\]
	where we use the fact that $\norm{(x-\overbar{a})}_1,\norm{(y-\overbar{b})}_1,$ and $\norm{(z-\overbar{c})}_1$ are each less than $r$. Then $DF(\overbar{a},\overbar{b},\overbar{c})$ is locally Lipschitz on $B_{r}(\overbar{a},\overbar{b},\overbar{c})$ with Lipschitz constant $2L$. 
	Now suppose $h \in \ell^1$ is a unit vector so we have
	\[
	[A_{(ij)}\cdot h]_{m,\alpha} = 
	\left\{
	\begin{array}{cc}
	[A_{(ij)}^{MN} \cdot h^{MN}]_{m,\alpha} & (m,\alpha) \prec (M,N) \\
	\frac{h_{m,\alpha}}{m} & (m,\alpha) \succ (M,N), i=j \\
	0 & \text{otherwise}
	\end{array}.
	\right.
	\]
	We let $\delta_i^j$ denote the Dirac delta so that we have the estimate 
	\begin{align*}
	\norm{A_{(ij)}}_1 & = \sup\limits_{\norm{h} = 1} \norm{
		\sum_{m=0}^{M} \sum_{\alpha=0}^{N} [A^{MN}_{(ij)}h^{MN}]_{m,\alpha} 
		+ \sum_{m=M+1}^{\infty} \sum_{\alpha=N+1}^{\infty} \delta_i^j \frac{1}{m} h_{m,\alpha}
	}_1 \\
	& \leq \sup\limits_{\norm{h} = 1} \norm{
		A^{MN}_{(ij)}h^{MN} + \delta_i^j \frac{1}{M} h^{\infty}
	}_1 \\
	& \leq \norm{A_{(ij)}}_1 \norm{h^{MN}}_1 + \delta_i^j \frac{\norm{h^{\infty}}_1}{M} \\
	& \leq \max \left\{\norm{A^{MN}_{(ij)}}_1, \delta_i^j \frac{1}{M} \right\}
	\end{align*}
	where we have used the fact that $\norm{h^{MN}}_1 + \norm{h^{\infty}}_1 = 1$. 
	Therefore, we conclude that
	\[
	\norm{A}_{\mathcal{X}} = \max \left\{ 
	\norm{A^{MN}}_{\mathcal{X}}, \frac{1}{M} 
	\right\}
	\]
	Taking these bounds together, if $\norm{(x,y,z) - (\overbar{a},\overbar{b},\overbar{c})}_{\mathcal{X}} \leq r$, we have the estimate
	\begin{align*}
	\norm{A(DF(x,y,z) - DF(\overbar{a},\overbar{b},\overbar{c}))}_{\mathcal{X}} & \leq \norm{A}_{\mathcal{X}} \norm{DF(x,y,z) - DF(\overbar{a},\overbar{b},\overbar{c})}_{\mathcal{X}} \\
	& \leq 2L \max \left\{\norm{A^{MN}}_{\mathcal{X}}, \frac{1}{M} \right\} \norm{(x,y,z) - (\overbar{a},\overbar{b},\overbar{c})}_{\mathcal{X}} \\
	&  \le Z_2r.
	\end{align*}

With these operators and bounds defined and equipped with Proposition \ref{prop:newton}, the validation for the advected image of a particular $\gamma$ parameterizing an arc in $\rr^3$ amounts to using a computer to rigorously verify that each of these estimates holds using interval arithmetic. The error bounds obtained as described in Remark~\ref{rmk:radPoly} are the sharpest bounds for which the contraction mapping theorem holds using these operators and bounds. If any of these bounds can not be verified or if the radii polynomial is non-negative, we say the validation fails. The a-posteriori nature of the validation yields little information about the cause for the failure and one must look carefully at the numerics, operators, bounds, or all three. 

\subsubsection*{Single step performance}
It is important to recognize that high precision numerics are not sufficient to control propagation error,
and we must carefully ``tune'' the parameters for the integrator in order to pass from 
deliberate and precise numerics to useful rigorous error bounds. To illustrate the importance 
of these tuning parameters as well as some heuristic methods for optimizing their values, we fix a benchmark arc segment, $\gamma_B$, to be the line segment between the equilibria 
in the Lorenz system with coordinates  
\[
p^{\pm} = (\pm \sqrt{\beta(\rho - 1)},\pm \sqrt{\beta(\rho - 1)},\rho - 1).
\]
Specifically, for the classical parameters $(\rho,\sigma,\beta) = (28, 10, \frac{8}{3})$ 
we will take our benchmark arc segment as
\[
\gamma_B(s) = 
\left(
\begin{array}{c}
0 \\
0 \\
27
\end{array}
\right)
+ 
\left(
\begin{array}{c}
\sqrt{72} \\
\sqrt{72} \\
0
\end{array}
\right)s. \qquad s \in [-1,1].
\]
This segment provides a reasonable benchmark, as it captures 
some of the worst behaviors we expect a vector field to exhibit. 
In particular, this segment does not align well with the flow, has 
sections which flow at vastly different velocities, and is a long arc
relative to the spatial scale of the system. 
That is, the length of the arc is the same order of the width of the attractor.
The combination of these bad behaviors make it a reasonable benchmark for 
showcasing heuristic subdivision and rescaling methods as well as the 
parameter tuning necessary for controlling error bounds.

For a typical manifold of initial conditions advected by a nonlinear vector field, error propagation in time is unavoidable and grows exponentially. Two natural strategies emerge when attempting to maximize reliability for long time integration. The first is to attempt to minimize precision loss from one time step to the next. This makes sense since the error carried forward from one time step contributes directly to the error in the next time step as initial uncertainty. The source of this error is primarily due to truncation, so that high order expansion in the spatial variables reduce the truncation error in the flow expansion. In other words, we are motivated to take $N$ as large as possible to control one-step error propagation. On the other hand, each time step incurs error which is unrelated to truncation error or decay of Taylor coefficients. The source of this error is simply due to the validation procedure which incurs roundoff errors as well as errors due to variable dependency from interval arithmetic. Therefore, we are simultaneously motivated to take fewer time steps. Evidently, this leads to a strategy which aims to maximize $\tau$ for a single time step. Recalling our estimate for the decay of the time coefficients in $\TT{\Gamma}$ it follows that maximizing $\tau$ implies we take $M$ as large as possible. 

The difficulty in carrying out both strategies is twofold. The obvious problem is computational efficiency. An examination of computation required for the validation of a single time step immediately reveals two operations which dominate the computational cost: the cost of numerically inverting the finite part of $A^{\dagger}$ as required for the definition of $A$, and the cost to form the matrix product, $A^{MN} DF^{MN}$, as required for the $Z_0$ bound. Both computations scale as a function of $M |N|$ which leads to a natural computational limitation on the effectiveness of either strategy. In fact, if the computational effort is fixed, say $M|N| = K$, then these strategies must compete with one another. Determining how to balance these competing strategies to obtain an overall more reliable parameterization is highly nontrivial even when $f,\gamma$ are fixed. For our benchmark segment we set $K = 1,777$ so that the matrix $A^{MN}$ has size $9K \approx 16,000$. Figure \ref{fig:compefficiency} illustrates the inherent trade-offs when attempting to balance $M$ and $N$ when $K= M|N|$ is fixed. 
\begin{figure}[t!]
	\includegraphics[width=.45\linewidth]{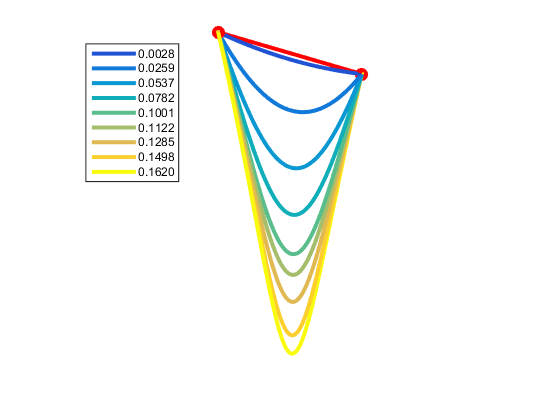}
	\includegraphics[width=.45\linewidth]{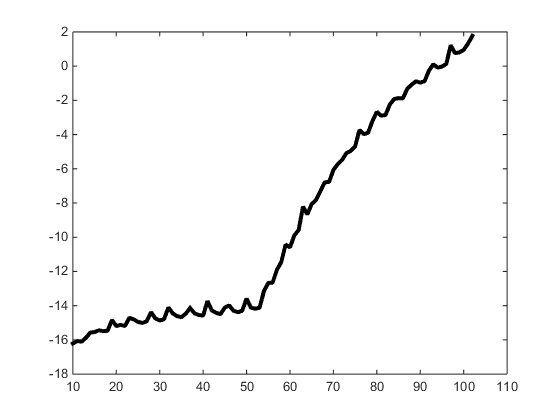}
	\caption{For fixed computational effort, $K = 1,777$: (left) The advected image of $\gamma_B$ (red line segment) for a single step with $M \in \{10,20,\dots,90\}$. The integration time resulting from our time rescaling is shown for each choice of $M$. (right) Single step error bounds (log10 scale) plotted against $M$. Initially, increasing $M$ has little effect on the error since $\gamma_B$ is of low order and $f$ is only quadratic. For $M > 55$, the increased precision loss due to truncation error becomes dramatic. For $M > 102$ the validation fails.} 
	\label{fig:compefficiency}
\end{figure}

A second difficulty arises in trying to balancing these two strategies which is more subtle. In this case we see that optimizing the choice for $M$ and $N$ typically depends heavily on both $f$ and $\gamma$. To illustrate how this occurs suppose $N \in \nn^{d-1}$ is a fixed spatial truncation and let $\epsilon = {M}/{|N|}$ and suppose $\overbar{a}$ is an approximation for $\TT{\Gamma}$ for a single time step. Recall from \eqref{eq:defF} that the coefficients of the form $[F(\overbar{a})]_{m,\alpha}$ are determined recursively from coefficients of the form $[\overbar{a}]_{j,\alpha}$ for $j \leq m-1$. Specifically, they are obtained by taking products in $\mathcal{X}$ which correspond to Cauchy products of Taylor series. We also recall that 
$\norm{F(\overbar{a})}_{\mathcal{X}}$ captures the truncation error and directly impacts the rigorous error bound as seen in the definition in 
Equation~\eqref{sec:Lorenz_Y0}. Evidently, if $\TT{\gamma}$ has nontrivial coefficients of higher order, these Cauchy products will produce nontrivial coefficients of even higher order. This effect occurs for each $1 \leq m \leq M$ with the Cauchy products constantly ``pushing'' weights into the higher order terms. This phenomenon in sequence space is a consequence of the geometric significance of these Taylor series. Typical polynomial parameterizations are rapidly deformed under advection and this stretching and compressing leads to analytic functions with nontrivial higher order derivatives. For ``large'' $\epsilon$, these nontrivial coefficients begin to contribute to the truncation error which is noticed in the $Y_0$ bounds. Moreover, the severity of this effect is determined by the order of the first ``large'' term in $\TT{\gamma}$, the parameter $\epsilon$, and the degree of the nonlinearity in $f$. 

As before, we considered our benchmark segment $\gamma_B$ and $\epsilon \in [.5,2]$. We further note that the Lorenz system is only quadratic and that $\gamma_B$ has no nonzero coefficients for $\alpha > 1$. Thus, the behavior indicated in Figure \ref{fig:epsilon} is driven exclusively by the tuning of $\epsilon$. 
\begin{figure}[h!]
	\includegraphics[width=1\linewidth]{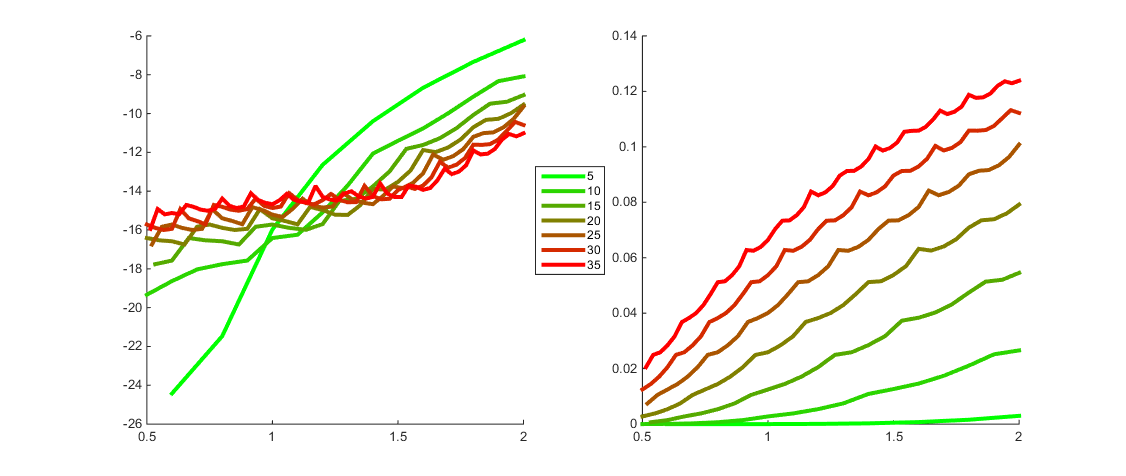}
	\caption{For fixed $N \in \{5,10,...,35\}$: (left) Single step error bounds (log10 scale) plotted against $\epsilon$. As $\epsilon$ increases, the loss in precision due to truncation error begins increasing dramatically which motivates taking $\epsilon$ small. (right) Single step integration time is plotted against $\epsilon$. Taking $\epsilon$ larger results in longer timesteps and thus requires fewer validations.}
	 \label{fig:epsilon}
\end{figure}
The severity of this effect for such mild choices of $f$ and $\gamma$ indicate that the long-time fidelity of our globalized manifold favors taking $\epsilon$ small. To say it another way, the precision loss in a single time step due to the validation is typically dominated by the truncation error. For our integrator, we conclude that tuning $\epsilon$ carefully is essential to controlling errors especially over long time intervals. 

\subsection{Long time advection and domain decomposition}
Regardless of our tuning of the integrator, the Taylor coefficients give a strict upper bound on the interval of time for which our expansion is valid. In this section, we describe our method for globalizing the local manifold by long time advection of boundary chart maps. First, we describe the decomposition of the time domain into subintervals. On each subinterval, we apply the single step algorithm to obtain its image under the flow valid on that subinterval, and this procedure is iterated to obtain the image of the local boundary for a longer interval in time. Next, we describe the necessity for spatial domain decomposition in between time steps for the partially advected arcs. This is a direct result of the nonlinear deformation experienced by a typical arc and we describe a rigorous decomposition algorithm.

\subsubsection*{Multiple time steps} \label{sec:time stepping}
No matter what choice is made for $(M,N)$, a single integration time step may not be sufficient for a practical application. For example, if one wants to propagate the local manifold for an interval of time which exceeds the radius of convergence (in time) of the Taylor expansion for some subset of points in $\gamma$. This necessarily requires one to extend the solution in time by analytic continuation. While this can be overcome to some degree, (e.g.\ by choosing a Chebyshev basis in the time direction) it can't be completely eliminated, especially if one is interested in growing the largest possible manifold. In this section we describe the method by which additional time steps can be computed with rigorous error estimates propagated from one time step to the next. As before, assume that $\gamma = \gamma_0$ is an analytic chart of the boundary of the local manifold where the zero subscript denotes the number of time steps of integration performed. Recall that we have $\gamma_0$ in the form of a polynomial of degree $N$ denoted by $\overbar{\gamma}_0$, and a rigorous analytic error estimate, $r_0$ such that the following inequality holds
\[
\norm{\TT{\overbar{\gamma}_0} - \TT{\gamma}}_{\mathcal{X}} < r_0.
\]
Our one-step integration algorithm takes input $(\overbar{\gamma}_0,r_0)$ and produces output of the form $(\overbar{\Gamma}_1,\tau_1,r_1)$ such that
\[
\norm{\TT{\overbar{\Gamma}_1} - \TT{\Phi(\gamma_0(s),t)}}_{\mathcal{X}} < r_1.
\]
holds for all $(s,t) \in [-1,1]\times[t_0,t_0 + \tau_1]$. Now, we define $\gamma_1(s) = \Gamma(s,\tau_1)$ which can be regarded as a polynomial plus analytic error bound of the form
\[
\gamma_1(s) = \overbar{\Gamma}(s,\tau_1) + r_1 = \overbar{\gamma_1}(s) + r_1.
\]
That is, $\gamma_1(s)$ is the evolved image of $\gamma_0$ under the time-$\tau_1$ flow map. Moreover, $\gamma_1 = \overbar{\gamma_1} + r_1$ has the appropriate form our one-step integrator and propagation to the next time step results by integrating $\gamma_1$. In other words, time stepping is performed by ``collapsing'' the $d$-dimensional output from one time step to the $(d-1)$-dimensional image of the time-$\tau_1$ map, and passing this as the input to the next time step. It follows that the advected image of $\gamma$ on an interval $[t_0,t_0 + T]$ can be obtained as a triplet of sequences: $\{r_0,\dots,r_l\}, \{\overbar{\gamma}_0,\dots,\overbar{\gamma}_l\}, \{t_0,\dots,\tau_{l-1},T\}$ for which at each step we have a rigorous estimate of the form
\[
\label{eq:piecewise bounds}
\norm{\TT{\overbar{\gamma}_i} - \TT{\Phi(\overbar{\gamma}_{i-1}(s),\tau_i)}}_{\mathcal{X}} < r_i \qquad 1 \leq i \leq l
\]
where $\TT{\overbar{\gamma}_i}$ is an approximation of the sequence of Taylor coefficients for $\Phi(\gamma(s),t)$ centered at $\tau_i$ with the expansion valid on the interval $[\tau_i-\tau_{i+1}, \tau_i+\tau_{i+1}]$. Evidently, each of the rigorous bounds in \eqref{eq:piecewise bounds} immediately implies the corresponding bound on the $C^0$ norm. Thus, we can define the piecewise polynomial 
\[
\overbar{\Gamma}(s,t) = \overbar{\gamma}_i(s,t-\tau_i) \quad \text{for} \ t \in [\tau_i,\tau_{i+1}]
\] 
and since the sequence of error bounds is nondecreasing (i.e.\ $r_{i+1} \geq r_i$) we have \\
$\norm{\overbar{\Gamma}(s,t) - \Phi(\gamma(s),t)}_{\infty} < r_l$ for all $(s,t) \in [-1,1]\times[t_0,t_0+T]$. 

\subsubsection*{Spatial domain decomposition}
\label{sec:remesh}
No matter how carefully one tunes the above parameters, surfaces will generically undergo deformation at exponential rates. Thus, despite any efforts at controlling the error propagation in a single time step, at some point the initial surface for a given time step will be excessively large. Attempting to continue integrating it results in a rapid loss of precision and a marked loss in integration time per step. Thus, typical manifolds of initial conditions can be advected for a short time, before requiring subdivision into smaller sub-manifolds. Performing this subdivision {\em rigorously} presents several challenges which must be addressed. We refer to this problem as the {domain decomposition} problem and we remark that a complete discussion is beyond the scope of this current work. However, our goal in this section is to describe pragmatic methods for efficiently estimating nearly optimal domain decompositions. In particular, we are interested in partially solving this problem by developing methods to answer three questions related to the general problem. When should a manifold be subdivided? How can this subdivision be done to maintain a mathematically rigorous parameterization of the global manifold? Finally, where are the best places to ``cut'' the manifold apart? In the remainder of this section, we address each of the questions. 

\subsubsection*{When to subdivide}
The first consideration is determining when an arc should be subdivided. Evidently, we are interested in subdividing any time our surface undergoes large scale deformation. However, this criterion is difficult to evaluate by evaluating surface area alone. The reason for this is that a surface can simultaneously have relatively small surface area and large higher order Taylor coefficients caused by cancellation. However, these large Taylor coefficients result in excessive error given our norm on $\mathcal{X}$. On the other hand, each time a surface is subdivided the computational effort required to propagate it increases exponentially. Thus, subdiving too conservatively will result in excessive computation times and the decision to subdivide is typically motivated by a particular error threshold required. Our solution which attempts to optimize this trade-off illustrates a powerful feature inherent in utilizing the radii polynomial method for our validation. Namely, by applying the Newton-Kantorovich theorem ``in reverse'', we are assured that if $\epsilon$ is chosen conservatively and the numerical approximation is close, that our error bound from the validation will be tight.

With this in mind it is natural to assign an acceptable precision loss for a given time step. This is normally done by prescribing a desired error bound on the final image and requiring each time step interval to incur loss of precision no greater than the average. A surface which exceeds this threshold is identified as defective and may be dealt with either by subdivision, decreasing $\epsilon$. or both. Given the difficulty in choosing $\epsilon$ for arbitrary surfaces and the necessity of subdivision eventually for any choice of truncation, we have chosen to always perform subdivision. In other words, we determine when a surface should be subdivided by performing the validation. If we don't like the error bound obtained, we subdivide the initial surface and integrate it again. Ultimately, the cost in utilizing this method is a single integration step for each subdivision which can be regarded as inexpensive next to the cost of subdividing a surface too early and performing exponentially more integrations over the course of the globalization procedure.

\subsubsection*{How to subdivide}
Next, we describe how a surface can be subdivided once one has determined the need to do so. Specifically, we are interested in {\em rigorous} subdivision of analytic surfaces which will require some care. We will describe our method in the context of the Lorenz example (i.e.\ $d = 2$) and note that the extension to higher dimensional surfaces is straightforward. Thus, we suppose $\gamma(s)$ is an analytic arc segment converging for $s \in [-1,1]$ and $\Gamma(s,t)$ its evolution under the flow with coefficient sequence $\mathcal{T}(\gamma) = a \in \mathcal{X}$. Subdivision of this arc amounts to choosing a subinterval, $[s_1,s_2] \subseteq [-1,1]$ and defining an appropriate transform $T: \mathcal{X} \rightarrow \mathcal{X}$ such that 
\[
\itt{T(a)} \left|_{[-1,1]} \right. = \itt{a} \left|_{[s_1,s_2]} \right..
\]
Moreover, when the rescaling is chosen to be linear we have $T \in \mathcal{L}(\mathcal{X})$. To make our rigorous rescaling precise, we define $s_{*} = \frac{s_1+s_2}{2}$ and $\delta = \frac{s_2-s_1}{2}$ so that computing the coefficients for $\Gamma$ recentered at $s_*$ and rescaled by $\delta$ is given by direct computation
\begin{align*}
\Gamma(\delta s,t)  = & \sum_{\alpha=0}^{\infty} a_\alpha(t)(\delta s)^\alpha \\
= & \sum_{\alpha=0}^{\infty} \delta^\alpha a_\alpha(t)(s - s_{*} + s_{*})^\alpha \\
= & \sum_{\alpha=0}^{\infty} \delta ^\alpha a_\alpha(t) \sum_{\kappa=0}^{\alpha} \binom{\alpha}{\kappa}s_{*}^{\alpha-\kappa} (s - s_{*})^\kappa \\
= & \sum_{\kappa=0}^{\infty} \sum_{\alpha=\kappa}^{\infty} \delta ^\alpha a_\alpha(t) \binom{\alpha}{\kappa}s_{*}^{\alpha-\kappa} (s - s_{*})^\kappa \\
= & \sum_{\kappa=0}^{\infty} c_\kappa(t)(s - s_{*})^\kappa 
\end{align*}
where each $a_\alpha(t)$ is an analytic scalar function of time and $c_\kappa(t) = \sum_{\alpha=\kappa}^{\infty} \delta^\alpha a_\alpha(t) \binom{\alpha}{\kappa}s_{*}^{\alpha-\kappa}$. 
Evidently, if $a \in \mathcal{X}$ and if $|\delta + s_*| \leq 1$, then $T(a) \in \mathcal{X}$. It follows that the coefficients for $T(a)$ are given explicitly by 
\[
[T(a)]_\alpha = \sum_{\kappa=\alpha}^{\infty} \delta^\kappa a_\kappa(t) \binom{\kappa}{\alpha}s_{*}^{\kappa-\alpha}
\]
and in particular, we note that $T$ is a linear operator on $\mathcal{X}$.

Now, we note that application of $T$ must be performed rigorously to preserve the error bounds on the global manifold. This requires controlling the error propagation induced by applying $T$ on the tail of $a$ as well as numerical precision loss from applying $T$ on the finite approximation. The key to controlling this error begins with the following estimate on the decay of $T(a^{(k)})$ for $1 \leq k \leq n$. For notational convenience, suppose for the moment that $a \in \ell^1_2$ is arbitrary, then we have:
\begin{align*}
\norm{T(a)}_1 = & \sum_{\alpha = 0}^{\infty} \left| c_\alpha \right| \\ 
= & \sum_{\alpha = 0}^{\infty} \sum_{\kappa = \alpha}^{\infty} \delta^\kappa \left| a_\kappa(t) \right| \binom{\kappa}{\alpha} s_*^{\kappa - \alpha}\\
= & \sum_{\kappa = 0}^{\infty} \sum_{\alpha = 0}^{\kappa} \delta^\kappa  \left| a_\kappa(t) \right| \binom{\kappa}{\alpha} s_*^{\kappa - \alpha}\\
= & \sum_{\kappa = 0}^{\infty}  \left| a_\kappa(t) \right| \sum_{\alpha = 0}^{\kappa} \delta^\kappa \binom{\kappa}{\alpha} s_*^{\kappa - \alpha}\\
= & \sum_{\kappa\kappa = 0}^{\infty}  \left| a_\kappa(t) \right| (\delta + s_*)^\kappa.\\
\end{align*}
and we recall that $\left| \delta + s_* \right| \leq 1$. Therefore $\norm{T(a)}_1 \leq \norm{a}_1$ always holds implying that the analytic error bound for the current step in time is automatically a bound for each submanifold. In fact, if $-1 < s_1 < s_2 < 1$, then we have the strict inequality $\left|\delta + s_*\right| < 1$ implying that our error estimate for the reparameterized manifold may {\em decrease} after subdivision. Actually, this is not surprising as our surfaces are analytic by assumption and thus the maximal principle applies.  

With these bounds on the tail error, we now consider the numerical precision loss occuring when $T$ is applied to the finite approximation given by 
\[
\overbar{a} \mapsto T({\overbar{a}}).
\]
This error is controlled similarly to other sources of floating point approximation error (i.e.\ by using interval arithmetic for each computation). Once interval enclosures of $T(\overbar{a})$ are obtained, we can pass once again to a floating point approximation for $T(\overbar{a})$ by applying the analytic shrink wrapping described in Section \ref{sec:intro} to each coefficient. 

\begin{figure}[h!]
	\includegraphics[width=.49\linewidth]{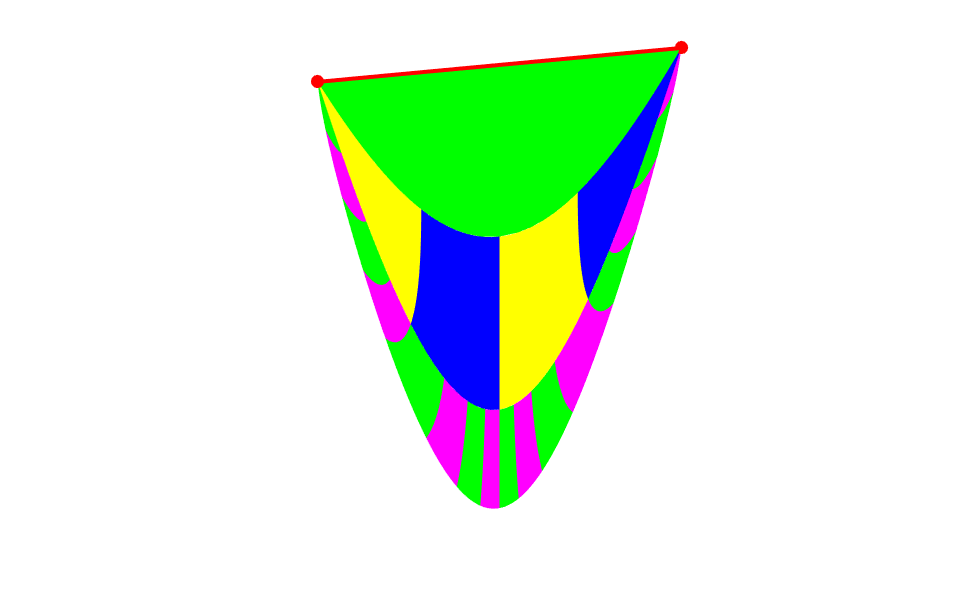}
	\includegraphics[width=.49\linewidth]{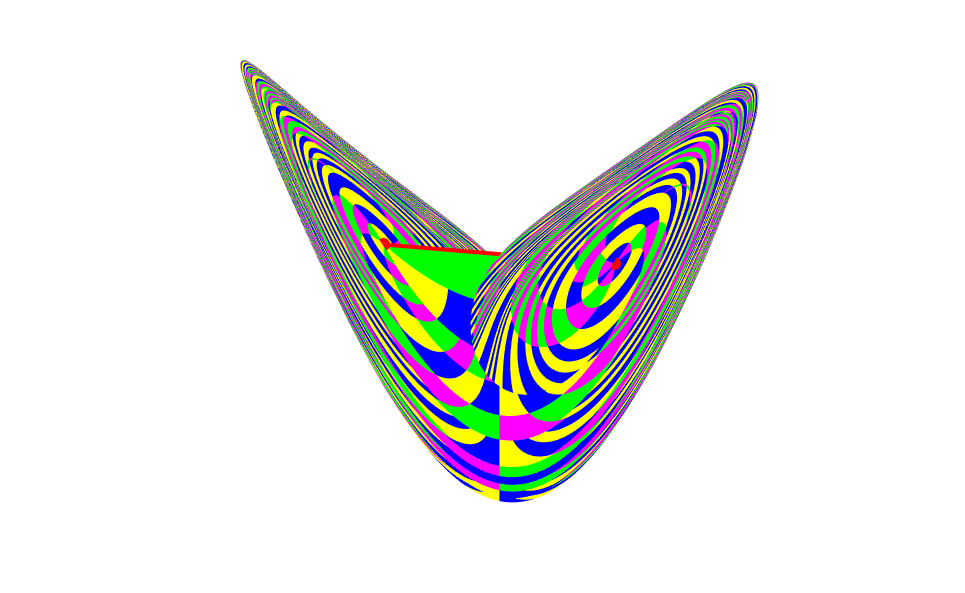}
	\caption{Integration of $\gamma_B$ (red segment) for fixed degree $(M,N) = (39,24)$. Each distinct color is the image of a single chart map. (left) 3 steps of integration in time with 4 subdivisions in space performed in between each. The total integration time is $\tau = .25$ units with error bound $1.1640 \times 10^{-13}$. (right) The forward image of the benchmark segment using our the automatic space/time subdivision schemes. After $\tau = 1$ units the surface is parameterized by 941 charts with a rigorous error bound of $6.6969 \times 10^{-13}$.}
	\label{fig:subdivscheme}
\end{figure}

\subsubsection*{Where to subdivide}
Our last consideration is to determine how to decompose the domain in such a way that each sub-manifold has somewhat similar error propagation. We note that it is too much to ask that the error propagation on each submanifold is identical. A more feasible goal is to describe a domain decomposition algorithm which is efficiently computable and which outperforms the naive decomposition. By naive decomposition, we are referring to subdividing using a uniformly spaced grid along each spatial dimension which tends to perform poorly based on experimentation. For comparison we note that such a decomposition amounts to uniform subdivision with respect to the box metric on $\rr^d$. Our domain decomposition scheme aims to define an alternative metric on $\rr^d$ which performs better. Intuitively, our goal is to choose a metric which defines initial conditions in a surface to be ``close'' if their trajectories do not separate under advection for some fixed time interval. Conversely, we want to define points on the manifold as far apart when their trajectories rapidly diverge from one another. 

To make this more precise let $x = \gamma(s)$ be a given surface and we begin by fixing a point, $x_0 = \gamma(s_0)$, and a time interval, $[t_0,t_0+T]$. For any $t \in [t_0,t_0+T]$ recall that $\Gamma(s,t)$ is a smoothly embedded $(d-1)$-manifold in $\rr^n$ and thus its pullback induces a well defined Riemannian metric. We will let $\rho_{t}(x_0,x)$ denote the geodesic distance between $x_0$ and $x$ along $\Gamma(s,t)$ with respect to this Riemannian metric. More generally, allowing $t$ to vary we note that $\Gamma(s,t)$ is a smoothly embedded $d$-manifold and we let $\rho(x_0,x)$ denote the associated geodesic distance. Now, we are interested in measuring the extent that the manifold is stretching locally (with respect to $\rho$) near $x_0$ on the time interval $[t_0,t_0+T]$. Specifically, let $B_R = \{s \in [-1,1]: \ \rho(\gamma(s),x_0) < R \}$ denote the open ball of radius $R$ centered at $x_0$ in the $\rho$-topology and define
\[
h_T(R,s_0) = \sup_{s \in B_R} \int_{t_0}^{t_0+T} \rho_t(x_0,\Gamma(s,t)) \ dt.
\] 
Our interest will be in evaluating a ``Lyapunov-exponent-like'' scalar defined by $\sigma(s_0) = \lim\limits_{R \rightarrow 0} h_T(R,s_0)$. Intuitively, $\sigma(s)$ measures the fitness of the local manifold with respect to integration. That is, relatively large values of $\sigma$ indicate the regions of the manifold which are expected to undergo the most rapid deformation over the time interval $[t_0,t_0+T]$. Based on this observation, our strategy for subdivision is to subdivide the manifold uniformly with respect to $\sigma$. 

It is important to note that we have no hope of evaluating $\sigma$ explicitly, however, in this instance a rigorous computation is not required. Rather, we need only to approximate uniform spacing in an efficient manner and we describe a fast numerical algorithm for this below. As with the previous sections, we will restrict to the case where $d=2$ and note that the the extension to higher dimensions is straight forward. 
\begin{enumerate}
\item Initiate grids of spatial and temporal sample points given by $\{s_0,s_1,\dots,s_I\},\{t_0,t_1,\dots,t_J\}$ respectively where $s_0 = -1,s_I = 1,$ and $t_J = t_0+T$. These grids may be uniform and many times more dense than the number of subdivisions. 
\item Use a numerical integrator (e.g.\ Runga Kutta) to compute $x_i(t_j)$ where $x_i(t_0) = \gamma(s_i)$. 
\item For each $0 \leq i \leq I$, approximate $h_T(R,s_i)$ by central differences in space:
\[
h_T(R,s_i) \approx \frac{1}{2} \int_{t_0}^{t_0+T} \rho_t(x_i(t),x_{i+1}(t)) + \rho_t(x_i(t),x_{i-1}(t)) \ dt.
\]
\item We next approximate the geodesic distance by the Euclidean distance:
\[
\rho_t(x_i(t),x_{i+1}(t)) \approx |x_i(t) - x_{i+1}(t)|.
\]
This yields a precise estimate for the geodesic distance when the spatial grid is very dense. This is not a limitation however given the efficiency of modern numerical integrators and the fact that we perform this on relatively small intervals in time. 
\item Finally, we are led to approximate these integrals by some quadrature method. As with the spatial grid, it is computational feasible to take a dense temporal grid so that the quadrature used makes little impact. The final result is an approximation of the form
\[
\sigma(s_i) \approx \frac{1}{2} \sum_{j=0}^{J} w_j(|x_i(t_j), - x_{i+1}(t_j)| + |x_i(t_j), - x_{i-1}(t_j)|)
\]
where $w_j$ is the quadrature weight. 
Extracting a uniform subgrid in space is performed on the approximated values of $\sigma$ on the dense spatial grid to any desired number of subdivisions. 
\end{enumerate}
This spatial subdivision scheme combined with our automatic time rescaling yields a reliable method for balancing efficient computation while minimizing error propagation. The results for our benchmark segment are given in Figure~\ref{fig:subdivscheme}.

\section{Results for Lorenz system}
\label{sec:Lorenzresults}
In this section we present additional details about the results obtained for the Lorenz system. Our main example has been the stable manifold at the origin for which we have globalized the manifold using two distinct strategies. In both cases we begin with a local parameterization of the stable manifold, $P$, of order $(M,N) = (100,100)$. The initial error for this local parameterization is rigorously bounded by $8.9743 \times 10^{-14}$. Next, a piecewise parameterization for the boundary of the local manifold is initially obtained by parameterizing 8 line segments with endpoints $s_1,s_2 \in \{-1,0,1\}$ and lifting each segment through $P$. Next, we advect these boundary arcs in (backward) time to grow the local manifold. The result in Figure~\ref{fig:faststable} shows the resulting manifold after $\tau = -.3$ time units. The initial local manifold is the dark blue patch. The error propagation for the globalized manifold and the number of charts for the parameterization is given in Table \ref{table:stabletable}.

\begin{figure}[!htb]
	\includegraphics[width=1.075\linewidth]{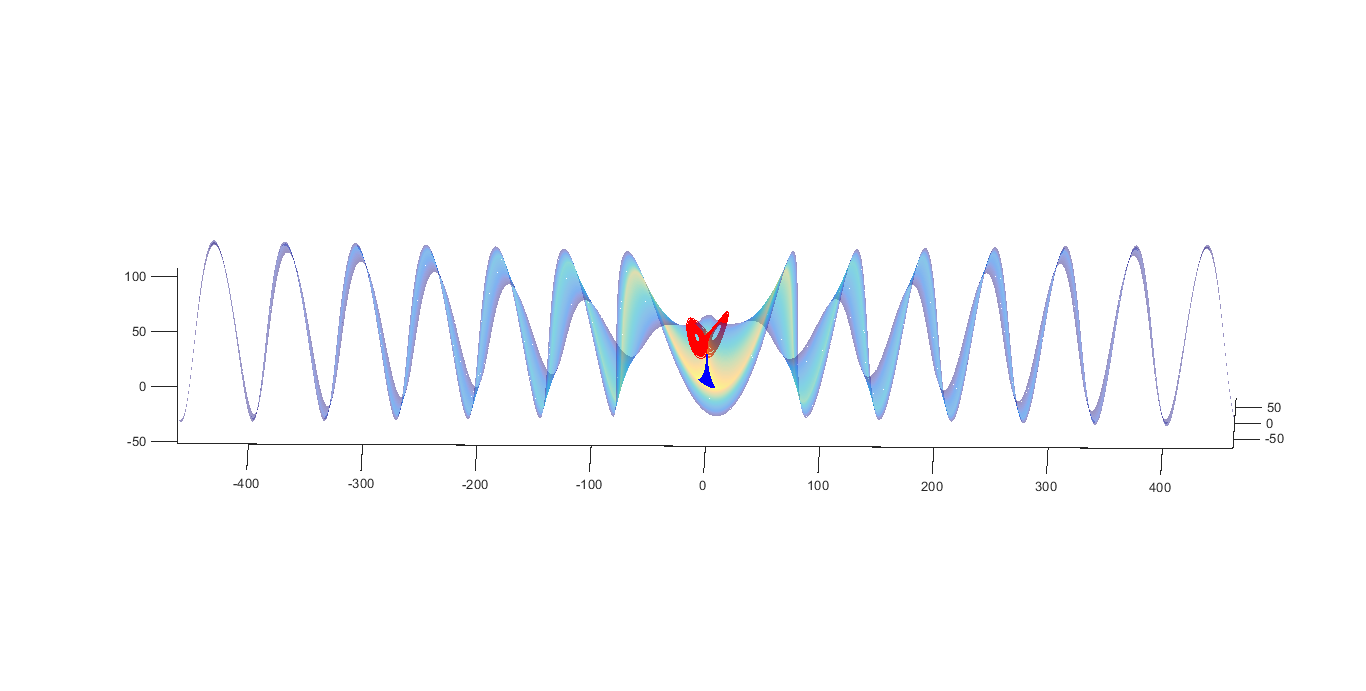}
	\caption{A validated two dimensional local stable 
		manifold of the origin in the Lorenz system:
		The initial local chart $P$ is obtained using the parameterization method, 
		as discussed in Section \ref{sec:parameterization}, and describes the 
		manifold in a neighborhood of the origin. The local stable manifold
		is the dark blue patch in the middle of the picture, below the attractor. A reference orbit near the attractor is shown in red for context. The boundary of the image of $P$ is meshed into arc segments and the global manifold is computed by advecting arcs by the flow using the rigorous integrator discussed in Section \ref{sec:rigorousIntegration}.} 
	\label{fig:faststable}
\end{figure}

\begin{table}[!htb]
\begin{center}
\begin{tabular}{|c|c|c|}
$\tau$ & Error Bound & Chart Maps \\
\hline 
0 	 & $8.9743 \times 10^{-14}$ & 8 \\
-0.1 & $3.2634 \times 10^{-13}$ & 88 \\ 
-0.2 & $1.7192 \times 10^{-7}$ & 396 \\ 
-0.3 & $2.9883 \times 10^{-7}$ & 746 \\ 
-0.4 & $2.9883 \times 10^{-7}$ & 1056 \\ 
-0.5 & $2.9883 \times 10^{-7}$ & 1374 \\ 
-0.6 & $1.0255 \times 10^{-6}$ & 1628 \\ 
-0.7 & $2.8063 \times 10^{-6}$ & 1906 \\ 
-0.8 & $7.7323 \times 10^{-6}$ & 2715 \\ 
-0.9 & $2.6827 \times 10^{-5}$ & 3615 \\ 
-1 & $1.0754 \times 10^{-4}$ & 4674 \\ 
\end{tabular}
\hspace{1cm}	
\begin{tabular}{|c|c|c|}
$\tau$ & Error Bound & Chart Maps \\
\hline 
0 	 & $8.9743 \times 10^{-14}$ & 8 \\
-0.03 & $1.1070 \times 10^{-13}$ & 12 \\ 
-0.06 & $1.3967 \times 10^{-13}$ & 38 \\ 
-0.09 & $2.2279 \times 10^{-13}$ & 88 \\ 
-0.12 & $3.0669 \times 10^{-13}$ & 164 \\ 
-0.15 & $5.0042 \times 10^{-13}$ & 276 \\ 
-0.18 & $7.7479 \times 10^{-13}$ & 446 \\ 
-0.21 & $1.3801 \times 10^{-12}$ & 702 \\ 
-0.24 & $2.9119 \times 10^{-12}$ & 1106 \\ 
-0.27 & $6.8347 \times 10^{-12}$ & 2006 \\ 
-0.3 & $1.6490 \times 10^{-11}$ & 5032 \\ 
\end{tabular}
\caption{Error propagation for the stable manifold. In both examples, the initial manifold boundary is parameterized by 8 arcs. (left) The slow (clipped) manifold shown in Figure~\ref{fig:slowstable}. (right) The unclipped manifold shown in Figure~\ref{fig:faststable}.}
\label{table:stabletable}
\end{center}
\end{table}

One immediately notices that the manifold rapidly expands away from the origin. This is unsurprising as the stable eignenvalues for the linearization at the origin are $\lambda_s \approx -2.67$ and $\lambda_{ss} \approx -22.83$. Our strategy to use the flow to globalize the local manifold naturally gives preference to the ``fast'' part of the local manifold. Therefore, we have globalized the local manifold by another strategy where we fix a compact subset, $K \subset \rr^3$, and grow the manifold in all directions until it exits $K$. For example, setting $K = [-100,100] \times [-100,100] \times [-40,120]$ we obtain the picture in Figure \ref{fig:slowstable} which we refer to as the {\em slow} manifold. As the fast manifold exits $K$, it is clipped using the domain decomposition algorithm and we continue advecting the remaining portion. The resulting error propagation is also shown in Table \ref{table:stabletable}. We note that once the slow manifold begins feeling the nonlinear dynamics, the manifold complexity becomes readily apparent. This complexity is seen in the repeated folding and shearing of the manifold shown in Figure \ref{fig:slowstable}. We also note that the appearance of additional connected components of the manifold is caused by this clipping procedure as portions of the manifold may exit $K$ only to return a short time later. From the numerical point of view, this complexity is also noticed as the total error propagation for the slow manifold is several orders of magnitude larger (per unit surface area).

\begin{figure}
\begin{floatrow}
\ffigbox{%
\includegraphics[width = \linewidth]{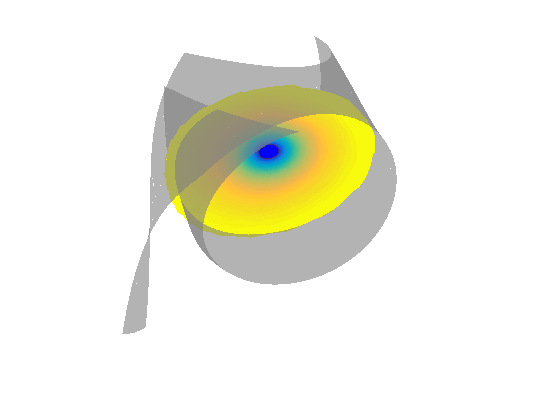}
}{%
\caption{The two-dimensional local unstable manifold at $p^+$ (central blue region) is rigorously computed and then extended by advecting its boundary by the flow (yellow) until it intersects the two-dimensional stable manifold of the origin (gray).}
\label{fig:connection}
}
\capbtabbox{%
\begin{tabular}{|c|c|c|}
$\tau$ & Error Bound & Chart Maps \\
\hline 
0 	 & $2.5271 \times 10^{-14}$ & 20 \\
2 & $5.5446 \times 10^{-13}$ & 504 \\ 
4 & $5.5471 \times 10^{-12}$ & 1067 \\ 
6 & $4.5824 \times 10^{-11}$ & 1655 \\ 
8 & $4.9037 \times 10^{-10}$ & 2267 \\ 
10 & $4.6123 \times 10^{-9}$ & 2922 \\ 
12 & $4.5241 \times 10^{-8}$ & 3602 \\ 
14 & $5.0631 \times 10^{-7}$ & 4326 \\ 
16 & $6.2147 \times 10^{-6}$ & 5124 \\ 
18 & $8.6529 \times 10^{-5}$ & 5988 \\ 
20 & $7.4806 \times 10^{-4}$ & 6820 \\ 
\end{tabular}
		}{%
\caption{Error propagation for the unstable manifold at $p^{+}$. \label{table:unstableerror}}
		}

\end{floatrow}
\end{figure}

Finally, we present a similar example for the two-dimensional unstable manifold at $p^{+}$. The local manifold is parameterized by a chart, $P$, with $(M,N) = (100,100)$ as before. The central blue patch in Figure \ref{fig:connection} shows $P(\mathbb{D}^2)$. Now, we must choose a piecewise parameterization of the boundary (in $\mathbb{D}^2$) and lift it through $P$ to obtain the boundary in $\rr^3$ for $W_{\mbox{\tiny loc}}^u(p^+)$. In this case the unstable eigenvalues are complex conjugates, $\lambda_u^{1,2} \approx .0940 \pm 10.1945i$, and we recall that the real unstable manifold is given by $P(z,\overbar{z})$ where $z \in \partial \mathbb{D}^1$. A natural choice for parameterizing $\partial \mathbb{D}^1$ is to use a complex exponential. However, to maintain control over truncation error when compositing with $P$ it is advantageous to choose polynomial parameterizations. Thus, we will instead lift the boundary a 20-gon inscribed in $\mathbb{D}^1$. Specifically, we choose $21$ nodes uniformly on $\partial \mathbb{D}^1$ of the form $s_j = (\cos \frac{\pi j}{10}, \sin \frac{\pi j}{10})$ for $j \in \{0,\dots,20\}$, and define
\[
\gamma_j(s) = \frac{(1-s)s_j + (1+s)s_{j+1}}{2} \qquad \text{for } s\in [-1,1] \qquad j \in \{0,\dots,19\}. 
\]
Then for each $j$, $P(\gamma_j(s), \overbar{\gamma_j(s)})$ lies in $W_{\mbox{\tiny loc}}^u(p^+)$. Moreover, $P$ conjugates the dynamics on $W_{\mbox{\tiny loc}}^u(p^+)$ with the linear dynamics and each $\gamma_j$ is transverse to the linear flow. Thus, we are assured that each line segment lifts to an arc which is transverse to the flow as required. 

Now, globalizing the unstable manifold follows in the same manner as for the stable manifold. For the picture shown in Figure \ref{fig:connection} the unstable manifold is the yellow region which was obtained by integrating for $\tau = 20$ time units. The propagation error and number of charts for this computation is given in Table \ref{table:unstableerror}. Combined with the globalized stable manifold at the origin, these computations are sufficient to detect an intersection with $W^s(p^0)$ as shown in Figure \ref{fig:connection}. This intersection is a connecting orbit from $p^+$ to $p^0$ and a method for validating its existence is presented in \cite{MR3207723}. Of course, methods for proving existence of connections which do not require computation of the global manifolds have been successfully taken up. However, the novelty in our example is the addition of a rigorous guarantee that the connect found is the {\em shortest} such connection. Indeed, analytic continuation of these manifolds can detect and prove existence of connections between these manifolds, or conversely, can be implemented to rule out the possibility of such connections.

\begin{remark}[Low order versus high order parameterization of local stable/unstable manifolds]
It is interesting to consider for a moment the role of the parameterization
method in the calculations just discussed.  In particular, what is the virtue of a
high order parameterization of the stable/unstable manifold? 

Recall that mathematically rigorous methods for obtaining computer assisted truncation error 
estimates for the linear approximation of the local stable/unstable manifold are 
developed in \cite{mirelesCNSNS_2014}, and in particular the 
Lorenz system is considered in Section $5$ of the work just cited.
There it was shown that the linear approximation of the 
stable manifold restricted to a neighborhood of radius $3.16 \times 10^{-8}$ 
about the origin enjoys an approximation errors less than $\delta = 5.51 \times 10^{-14}$.
This is comparable to (a bit smaller than) the error associated with the high order parameterizations used in 
the calculations above.% (actually this error for the linear approximation 
%is a little better than what was used above). 

Now recall that
the parameterization method provides the conjugacy between the linear 
and the nonlinear dynamics on the manifold.  Consider the slow stable direction 
with $\lambda_s \approx -2.66$.
Beginning with an initial condition $s_0 = 3.16 \times 10^{-8}$
on the slow stable eigenspace (so that the linear approximation has error 
bound as in the last paragraph), we integrate the 
linear system for $\tau = -7.5$ seconds to obtain a final condition 
with 
\[
s_f \approx s_0 e^{\lambda_s  \tau} = 14.58.
%s_f \approx s_0 e^{\lambda_s * \tau} = 14.58.
\] 
This is roughly the size of the initial manifold patch (in the slow direction) used
in the computations above.
So, beginning with the linear approximation and a starting error of roughly $10^{-14}$
it is necessary integrate more than seven and an half time units to obtain a 
representation of the stable manifold as good as the one we started with in 
Figure \ref{fig:slowstable}.

Even assuming that the errors accumulate more slowly when we integrate near the 
equilibrium point, and that fewer subdivisions are needed, % in order to propagate the 
%boundary of the linear approximation,
%assuming also that fewer subdivisions are needed in order to propagate the 
%boundary of the linear approximation than were needed to propagate the nonlinear 
%patch parameterized above --
it should nevertheless be the case that, after integrating the boundary of the linear 
approximation for seven time units, 
%we will have lost some (perhaps much)
perhaps much of the initial $10^{-14}$ accuracy will be lost.
%by the time 
%we grow 
%in obtaining a stable manifold as large as the initial patch shown in Figure \ref{fig:slowstable}.
%So as long as the computational time needed to propagate the linear approximation for 
%$7$ time units is larger than the time needed to compute the manifold to high order using the 
%parameterization method, the effort is justified (in practice we see that this ratio
%favors the high order parameterization substantially). 
%But even if this were not the case, 
A more quantitative comparison of the two methods would be nice to explore, 
%It would be nice to explore these back of the envelope calculations in greater detail, 
but we postpone it
%a more quantitative comparison 
to a future work. We expect that
%
%it is still worth while to 
%computing the parameterization 
%to high order if only because as soon as we start integrating our error bounds start accumulating.
using the linear approximation %as just described to reproduce Figure \ref{fig:slowstable}
to obtain a stable manifold as large as the initial patch shown in Figure \ref{fig:slowstable}
would result in worse final error bounds and in a more time consuming computation.  

%It would be nice to explore these back of the envelope calculations in greater detail, 
%but we postpone a more quantitative comparision to a future work.  
\end{remark}

\section{Acknowledgments}
The authors wish to thank two anonymous referees for 
carefully reading the submitted draft of the manuscript. 
Their numerous helpful suggestions and corrections
 improved the final published version.

\bibliographystyle{unsrt}
\bibliography{papers}

\begin{thebibliography}{10}

\bibitem{MR1976079}
X.~Cabr{\'e}, E.~Fontich, and R.~de~la Llave.
\newblock The parameterization method for invariant manifolds. {I}. {M}anifolds
  associated to non-resonant subspaces.
\newblock {\em Indiana Univ. Math. J.}, 52(2):283--328, 2003.

\bibitem{MR1976080}
X.~Cabr{\'e}, E.~Fontich, and R.~de~la Llave.
\newblock The parameterization method for invariant manifolds. {II}.
  {R}egularity with respect to parameters.
\newblock {\em Indiana Univ. Math. J.}, 52(2):329--360, 2003.

\bibitem{MR2177465}
X.~Cabr{\'e}, E.~Fontich, and R.~de~la Llave.
\newblock The parameterization method for invariant manifolds. {III}.
  {O}verview and applications.
\newblock {\em J. Differential Equations}, 218(2):444--515, 2005.

\bibitem{MR2821596}
Jan~Bouwe van~den Berg, J.~D. Mireles~James, Jean-Philippe Lessard, and
  Konstantin Mischaikow.
\newblock Rigorous numerics for symmetric connecting orbits: even homoclinics
  of the {G}ray-{S}cott equation.
\newblock {\em SIAM J. Math. Anal.}, 43(4):1557--1594, 2011.

\bibitem{JDMJ01}
J.~D. Mireles~James.
\newblock Polynomial approximation of one parameter families of (un)stable
  manifolds with rigorous computer assisted error bounds.
\newblock {\em Indagationes Mathematicae}, 26(1):225--265, 2015.

\bibitem{MR3207723}
Jean-Philippe Lessard, Jason~D. Mireles~James, and Christian Reinhardt.
\newblock Computer assisted proof of transverse saddle-to-saddle connecting
  orbits for first order vector fields.
\newblock {\em J. Dynam. Differential Equations}, 26(2):267--313, 2014.

\bibitem{MR1652147}
Martin Berz and Kyoko Makino.
\newblock Verified integration of {ODE}s and flows using differential algebraic
  methods on high-order {T}aylor models.
\newblock {\em Reliab. Comput.}, 4(4):361--369, 1998.

\bibitem{MR2644324}
A.~Wittig, M.~Berz, J.~Grote, K.~Makino, and S.~Newhouse.
\newblock Rigorous and accurate enclosure of invariant manifolds on surfaces.
\newblock {\em Regul. Chaotic Dyn.}, 15(2-3):107--126, 2010.

\bibitem{MR1930946}
Piotr Zgliczynski.
\newblock {$C^1$} {L}ohner algorithm.
\newblock {\em Found. Comput. Math.}, 2(4):429--465, 2002.

\bibitem{MR3281845}
Gianni Arioli and Hans Koch.
\newblock Existence and stability of traveling pulse solutions of the
  {F}itz{H}ugh-{N}agumo equation.
\newblock {\em Nonlinear Anal.}, 113:51--70, 2015.

\bibitem{HLM}
Allan Hungria, Jean-Philippe Lessard, and J.~D. Mireles~James.
\newblock Rigorous numerics for analytic solutions of differential equations:
  the radii polynomial approach.
\newblock {\em Math. Comp.}, 85(299):1427--1459, 2016.

\bibitem{BDLM}
Jan~Bouwe van~den Berg, Andr\'ea Desch\^enes, Jean-Philippe Lessard, and
  Jason~D. Mireles~James.
\newblock Stationary coexistence of hexagons and rolls via rigorous
  computations.
\newblock {\em SIAM Journal on Applied Dynamical Systems}, 14(2):942--979,
  2015.

\bibitem{fourierAutomaticDiff}
Jean-Philippe Lessard, J.~D. Mireles~James, and Julian Ransford.
\newblock Automatic differentiation for {F}ourier series and the radii
  polynomial approach.
\newblock {\em Phys. D}, 334:174--186, 2016.

\bibitem{MR3068557}
J.~D. Mireles~James and Konstantin Mischaikow.
\newblock Rigorous a-posteriori computation of (un)stable manifolds and
  connecting orbits for analytic maps.
\newblock {\em SIAM J. Appl. Dyn. Syst.}, 12(2):957--1006, 2013.

\bibitem{parmChristian}
J.~B. Van~den Berg, J.~D. Mireles~James, and Christian Reinhardt.
\newblock Computing (un)stable manifolds with validated error bounds:
  non-resonant and resonant spectra.
\newblock {\em Journal of Nonlinear Science}, 26:1055--1095, 2016.

\bibitem{maximeJPMe}
Maxime Breden, J.P. Lessard, and J.~D. Mireles~James.
\newblock Computation of maximal local (un)stable manifold patches by the
  parameterization method.
\newblock {\em Indagationes Mathematicae}, 27(1):340--367, 2016.

\bibitem{jayAMSnotes}
J.~D. Mireles~James.
\newblock Validated numerics for equilibria of analytic vector fields:
  invariant manifolds and connecting orbits.
\newblock {\em (To appear in AMS lecture notes - winter short course series)},
  pages 1--55, 2017.

\bibitem{MR3504856}
Tomoyuki Miyaji, Pawel Pilarczyk, Marcio Gameiro, Hiroshi Kokubu, and
  Konstantin Mischaikow.
\newblock A study of rigorous {ODE} integrators for multi-scale set-oriented
  computations.
\newblock {\em Appl. Numer. Math.}, 107:34--47, 2016.

\bibitem{MR3022075}
D.~Ambrosi, G.~Arioli, and H.~Koch.
\newblock A homoclinic solution for excitation waves on a contractile
  substratum.
\newblock {\em SIAM J. Appl. Dyn. Syst.}, 11(4):1533--1542, 2012.

\bibitem{MR2312532}
Kyoko Makino and Martin Berz.
\newblock Suppression of the wrapping effect by {T}aylor model-based verified
  integrators: the single step.
\newblock {\em Int. J. Pure Appl. Math.}, 36(2):175--197, 2007.

\bibitem{cnLohner}
Daniel Wilczak and Piotr Zgliczynski.
\newblock ${C}^n$-{L}ohner algorithm.
\newblock {\em Scheade Informaticae}, 20:9--46, 2011.

\bibitem{MR950221}
E.~Adams, D.~Cordes, and R.~Lohner.
\newblock Enclosure of solutions of ordinary initial value problems and
  applications.
\newblock In {\em Discretization in differential equations and enclosures
  ({W}eissig, 1986)}, volume~36 of {\em Math. Res.}, pages 9--28.
  Akademie-Verlag, Berlin, 1987.

\bibitem{MR904317}
Rudolf~J. Lohner.
\newblock Enclosing the solutions of ordinary initial and boundary value
  problems.
\newblock In {\em Computerarithmetic}, pages 255--286. Teubner, Stuttgart,
  1987.

\bibitem{MR1387154}
Rudolf~J. Lohner.
\newblock Computation of guaranteed enclosures for the solutions of ordinary
  initial and boundary value problems.
\newblock In {\em Computational ordinary differential equations ({L}ondon,
  1989)}, volume~39 of {\em Inst. Math. Appl. Conf. Ser. New Ser.}, pages
  425--435. Oxford Univ. Press, New York, 1992.

\bibitem{MR1626596}
Z.~Galias and P.~Zgliczy{\'n}ski.
\newblock Computer assisted proof of chaos in the {L}orenz equations.
\newblock {\em Phys. D}, 115(3-4):165--188, 1998.

\bibitem{MR2271217}
Daniel Wilczak.
\newblock The existence of {S}hilnikov homoclinic orbits in the {M}ichelson
  system: a computer assisted proof.
\newblock {\em Found. Comput. Math.}, 6(4):495--535, 2006.

\bibitem{MR2012847}
Tomasz Kapela and Piotr Zgliczy{\'n}ski.
\newblock The existence of simple choreographies for the {$N$}-body problem---a
  computer-assisted proof.
\newblock {\em Nonlinearity}, 16(6):1899--1918, 2003.

\bibitem{MR1961956}
Daniel Wilczak and Piotr Zgliczynski.
\newblock Heteroclinic connections between periodic orbits in planar restricted
  circular three-body problem---a computer assisted proof.
\newblock {\em Comm. Math. Phys.}, 234(1):37--75, 2003.

\bibitem{MR3032848}
Maciej~J. Capi{\'n}ski.
\newblock Computer assisted existence proofs of {L}yapunov orbits at {$L_2$}
  and transversal intersections of invariant manifolds in the {J}upiter-{S}un
  {PCR}3{BP}.
\newblock {\em SIAM J. Appl. Dyn. Syst.}, 11(4):1723--1753, 2012.

\bibitem{MR3443692}
Maciej~J. Capi\'nski and Anna Wasieczko-Zajac.
\newblock Geometric proof of strong stable/unstable manifolds with application
  to the restricted three body problem.
\newblock {\em Topol. Methods Nonlinear Anal.}, 46(1):363--399, 2015.

\bibitem{MR2551254}
Antoni Guillamon and Gemma Huguet.
\newblock A computational and geometric approach to phase resetting curves and
  surfaces.
\newblock {\em SIAM J. Appl. Dyn. Syst.}, 8(3):1005--1042, 2009.

\bibitem{MR3118249}
Gemma Huguet and Rafael de~la Llave.
\newblock Computation of limit cycles and their isochrons: fast algorithms and
  their convergence.
\newblock {\em SIAM J. Appl. Dyn. Syst.}, 12(4):1763--1802, 2013.

\bibitem{doi:10.1137/140960207}
Roberto Castelli, Jean-Philippe Lessard, and J.~D. Mireles~James.
\newblock Parameterization of invariant manifolds for periodic orbits {I}:
  {E}fficient numerics via the {F}loquet normal form.
\newblock {\em SIAM J. Appl. Dyn. Syst.}, 14(1):132--167, 2015.

\bibitem{MR2679365}
Gianni Arioli and Hans Koch.
\newblock Computer-assisted methods for the study of stationary solutions in
  dissipative systems, applied to the {K}uramoto-{S}ivashinski equation.
\newblock {\em Arch. Ration. Mech. Anal.}, 197(3):1033--1051, 2010.

\bibitem{dlLFGL}
Rafael de~la Llave, Jordi-Llu{\'{\i}}s Figueras, Marcio Gameiro, and
  Jean-Philippe Lessard.
\newblock Theoretical results on the numerical computation and a-posteriori
  verification of invariant objects of evolution equations.
\newblock {\em In preparation}., 2014.

\bibitem{jpPO_PDE}
Marcio Gameiro and Jean-Philippe Lessard.
\newblock A posteriori verification of invariant objects of evolution
  equations: periodic orbits in the {K}uramoto-{S}ivashinsky {PDE}.
\newblock {\em SIAM J. Appl. Dyn. Syst.}, 16(1):687--728, 2017.

\bibitem{mamotreto}
\`Alex Haro, Marta Canadell, Jordi-Llu\'\i~s Figueras, Alejandro Luque, and
  Josep-Maria Mondelo.
\newblock {\em The parameterization method for invariant manifolds}, volume 195
  of {\em Applied Mathematical Sciences}.
\newblock Springer, [Cham], 2016.
\newblock From rigorous results to effective computations.

\bibitem{wittigThesis}
A.~Wittig.
\newblock {\em Rigorous high-precision enclosures of fixed points and their
  invariant manifolds}.
\newblock PhD thesis, Michigan State University, 2011.

\bibitem{ourCode}
Shane Kepley.
\newblock Invariant manifold software.
\newblock
  \url{http://cosweb1.fau.edu/~jmirelesjames/analyticContinuationPage.html},
  2017.

\bibitem{MR2338393}
Sarah Day, Jean-Philippe Lessard, and Konstantin Mischaikow.
\newblock Validated continuation for equilibria of {PDE}s.
\newblock {\em SIAM J. Numer. Anal.}, 45(4):1398--1424 (electronic), 2007.

\bibitem{MR2443030}
Jan~Bouwe van~den Berg and Jean-Philippe Lessard.
\newblock Chaotic braided solutions via rigorous numerics: chaos in the
  {S}wift-{H}ohenberg equation.
\newblock {\em SIAM J. Appl. Dyn. Syst.}, 7(3):988--1031, 2008.

\bibitem{MR2630003}
Jan~Bouwe van~den Berg, Jean-Philippe Lessard, and Konstantin Mischaikow.
\newblock Global smooth solution curves using rigorous branch following.
\newblock {\em Math. Comp.}, 79(271):1565--1584, 2010.

\bibitem{MR2776917}
Marcio Gameiro and Jean-Philippe Lessard.
\newblock Rigorous computation of smooth branches of equilibria for the three
  dimensional {C}ahn-{H}illiard equation.
\newblock {\em Numer. Math.}, 117(4):753--778, 2011.

\bibitem{MR2299977}
A.~Haro and R.~de~la Llave.
\newblock A parameterization method for the computation of invariant tori and
  their whiskers in quasi-periodic maps: explorations and mechanisms for the
  breakdown of hyperbolicity.
\newblock {\em SIAM J. Appl. Dyn. Syst.}, 6(1):142--207 (electronic), 2007.

\bibitem{MR2240743}
{\`A}.~Haro and R.~de~la Llave.
\newblock A parameterization method for the computation of invariant tori and
  their whiskers in quasi-periodic maps: numerical algorithms.
\newblock {\em Discrete Contin. Dyn. Syst. Ser. B}, 6(6):1261--1300
  (electronic), 2006.

\bibitem{MR2289544}
A.~Haro and R.~de~la Llave.
\newblock A parameterization method for the computation of invariant tori and
  their whiskers in quasi-periodic maps: rigorous results.
\newblock {\em J. Differential Equations}, 228(2):530--579, 2006.

\bibitem{fastSlow}
J.~B. van~den Berg and J.~D. Mireles~James.
\newblock Parameterization of slow-stable manifolds and their invariant vector
  bundles: theory and numerical implementation.
\newblock {\em Discrete Contin. Dyn. Syst.}, 36(9):4637--4664, 2016.

\bibitem{parmPDE}
J.~D. Mireles~James and Christian Reinhardt.
\newblock Fourier-{T}aylor parameterization of unstable manifolds for parabolic
  partial differential equations: Formalism, implementation, and rigorous
  validation.
\newblock {\em (Submitted) {\footnotesize
  \verb|http://cosweb1.fau.edu/~jmirelesjames/unstableManParmPDEPage.html|}},
  2016.

\bibitem{maximePOmanifolds}
Maxime Murry, J.P. Lessard, and J.~D. Mireles~James.
\newblock Computer assisted proof of transverse cycle-to-cycle connecting
  orbits for first order vector fields.
\newblock {\em (In preperation))}, 2016.

\bibitem{MR2728178}
J.~D. Mireles~James and Hector Lomel{\'{\i}}.
\newblock Computation of heteroclinic arcs with application to the volume
  preserving {H}\'enon family.
\newblock {\em SIAM J. Appl. Dyn. Syst.}, 9(3):919--953, 2010.

\bibitem{MR2652784}
Siegfried~M. Rump.
\newblock Verification methods: rigorous results using floating-point
  arithmetic.
\newblock {\em Acta Numer.}, 19:287--449, 2010.

\bibitem{Ru99a}
{S.M.} Rump.
\newblock {INTLAB - INTerval LABoratory}.
\newblock In Tibor Csendes, editor, {\em {Developments~in~Reliable Computing}},
  pages 77--104. Kluwer Academic Publishers, Dordrecht, 1999.
\newblock http://www.ti3.tu-harburg.de/rump/.

\bibitem{mirelesCNSNS_2014}
J.~D. Mireles~James.
\newblock Computer assisted error bounds for linear approximation of (un)stable
  manifolds and rigorous validation of higher dimensional transverse connecting
  orbits.
\newblock {\em Communications in Nonlinear Science and Numerical Simulation},
  2014.

\end{thebibliography}

\end{document}